\documentclass{elsarticle} 

\usepackage{amsthm,amsmath,verbatim,amsfonts,enumitem}

\usepackage{natbib}
\usepackage{hyperref}
\hypersetup{
colorlinks=true,
    linkcolor=blue,
    }

\journal{European Journal of Combinatorics}

\newtheoremstyle{noperiod}
{3pt}
{3pt}
{\em}
{}
{\bf}
{}
{.2em}
{}

\newtheorem{lemma}{Lemma}
\newtheorem{prop}{Proposition}

\newtheorem{cor*}{Corollary}[prop]
\newtheorem{obs}{Observation}
\newtheorem{thm}{Theorem}
\newtheorem{defn}{Definition}
\newtheorem{fact}{Fact}[section]
\newtheorem{cor}{Corollary}
\newtheorem{corfact}{Corollary}[fact]
\newtheorem{claimprop}{Claim}[prop]
\newtheorem{claimlemma}{Claim}[lemma]
\newtheorem*{prop*}{Proposition}
\newtheorem*{thm*}{Theorem}
\newtheorem*{lemma*}{Lemma}

\theoremstyle{noperiod}
\newtheorem*{introfact}{Fact}

\theoremstyle{definition}
\newtheorem{claimpropproof}{Proof of Claim}[prop]

\theoremstyle{definition}
\newtheorem{claimlemmaproof}{Proof of Claim}[lemma]

\DeclareMathOperator{\Aut}{Aut} 
\DeclareMathOperator{\diam}{diam}

\begin{document}

\begin{frontmatter}

\title{Twists and Twistability\footnote{February 2018}}

\author{Rebecca Coulson}
\address{Rutgers University, Department of Mathematics \\ 110 Frelinghuysen Road, Piscataway, NJ 08854}

\begin{abstract}
Metrically homogeneous graphs are connected graphs which, when endowed with the path metric, are homogeneous as metric spaces. In this paper we introduce the concept of \textit{twisted automorphisms}, a notion of isomorphism up to a permutation of the language. 
We find all permutations of the language which are associated with twisted automorphisms of metrically homogeneous graphs. For each non-trivial permutation of this type we also characterize the class of metrically homogeneous graphs which allow a twisted isomorphism associated with that permutation. The permutations we find are, remarkably, precisely those found by Bannai and Bannai in an analogous result in the context of finite association schemes \cite{BB}.
\end{abstract}

\begin{keyword}
Metrically homogeneous graphs, automorphisms, permorphisms, twists, language permutations
\MSC[2010] 03C50 \sep 05C76   
\end{keyword}

\end{frontmatter}




\section{Introduction}

Metrically homogeneous graphs are connected graphs which, when endowed with the path metric, are homogeneous as metric spaces. Moss in \cite{Mos-DG} and Cameron more explicitly in \cite{Cam-Cen} call for a full classification of metrically homogeneous graphs. Cherlin has proposed a conjectured classification, detailed in \cite{Che-HOGMH}.

In studying these graphs, we define a notion of isomorphism up to a permutation of the language---a notion we call ``twisted isomorphism."

Twisted isomorphisms have been considered in a number of model theoretic contexts, namely, in \cite{Her-EPIF,Her-EPI,Bag-TCT,Iva-AAG,JZ,Barb}, and \cite{CaT-AmRG}. In some of these articles, the authors refer to twisted isomorphisms as \textit{permorphisms}. Cameron and Tarzi in \cite{CaT-AmRG} also examine the group of twisted automorphisms of a structure. More recently in {\cite[\S 4.1]{ACM-MH3}} and {\cite[\S 2.2]{Che-MH4}} unexpected twisted isomorphisms between ostensibly different metrically homogeneous graphs played a useful role in classification theorems.

In section \ref{sec:twists} of this paper, we find all possible permutations of the language which transform some metrically homogeneous graph into another metrically homogeneous graph. In section \ref{sec:twistables} we analyze in each case which pairs of metrically homogeneous graphs are ``twistable" to each other by a twisted automorphism affording the specified permutation of the language.

Our main results are as follows, where Proposition \ref{Prop:Twist:Nongeneric} deals with special cases of the problem, and Theorems \ref{thm:twists} and \ref{Thm:Twistable:Graphs} deal with more typical cases. For a structure $\Gamma$ and a permutation $\sigma$ of the language of $\Gamma$, the symbol $\Gamma^{\sigma}$ denotes the structure in the same language in which the symbol $R$ denotes the relation which is denoted by $\sigma^{-1}(R)$ in $\Gamma$. When $\Gamma$ is a metrically homogeneous graph, the language is identified with the set of all distances which occur (cf. \S \ref{sec:basics}). Additional relevant definitions are also provided in \S \ref{sec:basics}.

\hrulefill

\begin{prop*}[\ref{Prop:Twist:Nongeneric}]
Let $\Gamma$ be a metrically homogeneous graph of non-generic type
or diameter $\delta\leq 2$,
and $\sigma$ a non-trivial permutation of the language such that 
$\Gamma^\sigma$ is metrically homogeneous.
Then $\Gamma^\sigma$ is also of non-generic type or diameter $\delta\leq 2$, and one of the following applies.
\begin{itemize}
\item $\Gamma$ has diameter 2 and is not complete multipartite, with $\sigma$ the transposition $(1,2)$.
\item $\Gamma$  is finite, antipodal of diameter $3$, not bipartite, 
with $\sigma$ the transposition $(1,2)$.
\item $\Gamma$ is an $n$-cycle $C_n$, $n\geq 7$; $\delta=\lfloor n/2\rfloor$: 
$\sigma$ is given by multiplication by $\pm k \pmod n$ for some $k$ with $(k,n)=1$, 
with $\pm i$ identified.
\end{itemize}
\end{prop*}

\begin{thm*}[\ref{thm:twists}]
Let $\sigma$ be a non-trivial permutation of the language of a metrically homogeneous graph $\Gamma$ of generic type where $\Gamma^{\sigma}$ is itself a metrically homogeneous graph. Then $\sigma$ is one of $\rho, \rho^{-1}, \tau_0, \tau_1$, which are defined as follows:

\begin{align*}
\rho(i)&=
\begin{cases} 2i & \mbox{$i\leq \delta/2$}\\
2(\delta-i)+1&\mbox{$i>\delta/2$}
\end{cases}&
\rho^{-1}(i)&=
\begin{cases} i/2 & \mbox{$i$ even}\\
\delta-\frac{i-1}{2}&\mbox{$i$ odd}
\end{cases}
\end{align*}
and for $\epsilon=0$ or $1$, $\tau_{\epsilon}$ is the involution defined by
\begin{align*}
\tau_\epsilon(i)&=
\begin{cases}
(\delta+\epsilon)-i&\mbox{for $\min(i,(\delta+\epsilon)-i)$ odd}\\
i &\mbox{otherwise}
\end{cases}\\
\end{align*}
\end{thm*}

\begin{thm*}[\ref{Thm:Twistable:Graphs}]
Let $\sigma$ be one of the permutations $\rho,\rho^{-1}, \tau_0$, or $\tau_1$, with $\delta\geq 3$. Then the metrically homogeneous graphs $\Gamma$ of generic type whose images $\Gamma^{\sigma}$ are also metrically homogeneous are precisely those with the numerical parameters $K_1, K_2, C, C'$ as in the table shown in section \ref{sec:twistables}.

\end{thm*}

\hrulefill

\bigskip

To establish the characterization of twistable metrically homogeneous graphs without assuming they are of known type requires some extension of the known structure theory for metrically homogeneous graphs, which is given in Proposition \ref{Prop:Realize:i,i,2k}. This may find further applications to the general classification problem for such graphs.

We succeed in classifying the twistable metrically homogeneous graphs explicitly in terms of associated numerical invariants to be discussed in section \ref{sec:twistables}. The question of their classification up to isomorphism is open, but modulo our results it is simply a small part of the larger classification problem, for very particular values of these numerical parameters.

We plan to examine twisted automorphism groups further in future work, along lines suggested by Cameron and Tarzi in the context of a natural generalization of the random graph. In particular, we explore the problems as to when the twisted automorphism group splits over the automorphism group, and when the automorphism group of the automorphism group is induced by the twisted automorphism group. That will be joint work with Gregory Cherlin.

We thank Peter Cameron for drawing our attention to the parallel with \cite{BB}.

\section{Background information}\label{sec:basics}

Here we provide definitions as well as known results about metrically homogeneous graphs which will be useful in understanding the content of this paper. The facts given here are taken from the Appendix, where additional facts utilized in our argument are found. Much more information about metrically homogeneous graphs can be found in \cite{Che-HOGMH}.

\begin{defn}
Let $\mathcal{M} = (M,(R^{\mathcal{M}})_{R \in \Lambda}$ be a relational structure with universe $M$, language $\Lambda$, and with the relation $R^{\mathcal{M}}$ corresponding to the relation symbol $R \in \Lambda$. Let $\sigma$ be a permutation of $\Lambda$.
\begin{itemize}
    \item $\mathcal{M}^{\sigma}$ is the structure with universe $M$ and with relations determined as follows:
    \begin{equation*}
        (R^{\sigma})^{\mathcal{M}^{\sigma}} = R^{\mathcal{M}}
    \end{equation*}
    \item A \emph{$\sigma$-isomorphism} between two structures $\mathcal{M}_1,\mathcal{M}_2$ with language $\Lambda$ is an isomorphism of $\mathcal{M}_1$ with $\mathcal{M}_2^{\sigma}.$
    \item A \emph{twisted isomorphism} between two structures with the same language is a $\sigma$-isomorphism for some permutation $\sigma$ of the language.
\end{itemize}

When the structures are equal we speak of \emph{$\sigma$-automorphisms} and \emph{twisted automorphisms}. The twisted automorphisms of a structure form a group, and the associated permutations of the language also form a group. We often refer to the permutation of the language associated to a given twisted automorphism as a \emph{twist} of the language. In particular the group of twisted automorphisms induces a group of twists.
\end{defn}


For example, a graph $\mathcal{G}$ may be viewed as having the language $\{V,E,E^c\}$, where $V$ is the vertex set and $E, E^c$ are edge and non-edge, respectively. Then a non-trivial twist of $\mathcal{G}$ would transpose $E$ and $E^c$. Thus a non-trivial twist acting on the language of $C_4$  would send it to two copies of $C_2$, while a non-trivial twist on the language of $C_5$ takes it to an isomorphic graph.

A graph $\mathcal{G}$ is \textit{homogeneous} if every isomorphism of finite induced subgraphs can be extended to a full automorphism of $\mathcal{G}$.

A connected graph may be viewed as a metric space by endowing it with the path metric. Thus we define a \textit{metrically homogeneous graph} $\Gamma$ to be a connected graph where every isometry between finite subspaces can be extended to a full isometry of $\Gamma$. Note that every connected homogeneous graph is a metrically homogeneous graph, where we are only concerned with the distances $d = 1$ and $2$.

We view a (connected) metrically homogeneous graph $\Gamma$ as having the language $(V,R_1,R_2, \cdots , R_{\delta})$ where $V$ is the vertex set of $\Gamma$, $\delta$ is the diameter of $\Gamma$, and 
\begin{equation*}
    R_i(v_1,v_2) \Leftrightarrow d(v_1,v_2) = i
\end{equation*}
where $d$ is the path metric.

Thus a twist of the language of a metrically homogeneous graph $\Gamma$ may be viewed as a permutation in $S_{\delta}$, with the diameter $\delta$ being permitted to be countably infinite.

For a metrically homogeneous graph $\Gamma$, we define $\Gamma_i$ to be the set of vertices at distance $i$ from some fixed basepoint. For our purposes, we view $\Gamma_i$ as a metric space with the induced metric. The homogeneity of $\Gamma$ ensures that $\Gamma_i$ is well-defined up to isomorphism. 

We say that the graph $\Gamma$ is of \textit{generic type} if any two vertices at distance $2$ have infinitely many common neighbors which themselves form an independent set and if $\Gamma_1$ does not admit any non-trivial $\Aut(\Gamma_1)$-invariant equivalence relation. 

All the known metrically homogeneous graphs of generic type are parametrized by six parameters $(\delta,K_1,K_2,C,C',\mathcal{S})$ (\cite{Che-HOGMH, ACM-MH3}), detailed below, which constrain possible metric configurations embedding in some graph.
We write $\mathcal{A}^{\delta}_{K_1,K_2,C_0,C_1,\mathcal{S}}$ for the class of finite metrically homogeneous graphs which satisfy these restrictions. When this class is an amalgamation class, its Fra\"{i}ss\'{e} limit is denoted $\Gamma^{\delta}_{K_1,K_2,C_0,C_1,\mathcal{S}}$ (for more on Fra\"{i}ss\'{e} limits see, among others, \cite{hod}. These parameters may be defined for any metrically homogeneous graph as follows.

\begin{defn}\label{defn:numpar}
Let $\Gamma$ be a metrically homogeneous graph of generic type and of diameter at least $3$. The parameters $(\delta,K_1,K_2,C,C',\mathcal{S})$ are defined as follows.
\begin{enumerate}
    \item $\delta$ is the diameter of $\Gamma$.
    \item $K_1$ is the least $k$ such that $\Gamma$ contains a triangle of type $(k,k,1)$, and $K_2$ is the greatest such. We take $K_1 = \infty, K_2 = 0$ if no such triangle occurs.
    \item $C_0$ is the least even number greater than $2\delta$ such that no triangle of perimeter $C_0$ is in $\Gamma$, and $C_1$ is the least such odd number.
    \item $C = \min(C_0,C_1)$ and $C' = \max(C_0,C_1)$.
    \item $\mathcal{S}$ is the collection of $(1,\delta)$-spaces $\mathcal{S}$ such that 
    \begin{itemize}
        \item $\mathcal{S}$ is forbidden, i.e. it does not embed in $\Gamma$; and
        \item $\mathcal{S}$ is not a forbidden triangle of $\Gamma^{\delta}_{K_1,K_2,C_0,C_1}$, and does not contain a smaller forbidden $(1,\delta)$-space.
    \end{itemize}
\end{enumerate}
\end{defn}

A more nuanced discussion of these parameters can be found in \cite{Che-HOGMH} and \cite{ACM-MH3}. 

There are two classes of metrically homogeneous graphs graphs which often require specialized reasoning---these are the \emph{bipartite} graphs and the \emph{antipodal} graphs. These are both of generic type. We will assume that the reader is familiar with bipartite graphs, and we will define antipodal graphs and describe some of their properties.

Let $\Gamma$ be a metrically homogeneous graph of finite diameter $\delta$. Then $\Gamma$ is \textit{antipodal} if for every vertex $v$ in $\Gamma$  there exists a unique $v' \in \Gamma$ such that $d(v,v') = \delta$.

Antipodal metrically homogeneous graphs are quite structured, as seen by the following two facts.

\begin{introfact}{\normalfont \ref{fact:antipodallawantipodalgraphs} {\cite[Theorem 11]{Che-2P}}}{\bf.}
Let $\Gamma$ be a metrically homogeneous graph and antipodal of diameter $\delta \geq 3$. Then for each $u \in \Gamma$ there exists a $u' \in \Gamma$ at distance $\delta$ from $u$ such that
\begin{equation*}
    d(u,v) = \delta - d(u',v)
\end{equation*}
for every $v \in \Gamma$.

In particular, the map $u \mapsto u'$ is a central involution of $\Aut(\Gamma)$.
\end{introfact}

We may also deduce the following.

\begin{cor}\label{fact:antipodaladdstodelta}
Let $\Gamma$ be a metrically homogeneous graph and antipodal of diameter $\delta \geq 3$. Then $K_1 + K_2 = \delta$.
\end{cor}

\begin{introfact}{\normalfont \ref{fact:antipodalSmallC} {\cite[Lemma 6.1]{Che-2P}}}{\bf.}
Let $\Gamma$ be a metrically homogeneous graph of diameter $\delta$. Then $\Gamma$ is antipodal if and only if no triangle has perimeter greater than $2\delta$.
\end{introfact}

Thus for an antipodal graph $\Gamma$, the parameters $C,C'$ must be
\begin{align*}
    C = C_1 = 2\delta + 1  && C' = C_0 = 2\delta + 2.
\end{align*}

Bipartite graphs $\Gamma$ have no triangles of odd perimeter, and therefore their parameters satisfy
\begin{align*}
    K_1 = \infty && K_2 = 0 && C = C_1 = 2\delta + 1.
\end{align*}

We define a \textit{metric triangle} with side lengths $i,j,k$ to be a triple of points $u,v,w$ whose pairwise path distances are $i,j,$ and $k$ respectively. Such a metric triangle is said to have \textit{triangle type} $(i,j,k)$, with no particular order imposed on $i,j$ and $k$. We say that the triangle type $(i,j,k)$ is \textit{realized in $\Gamma$} if there is a metric triangle with triangle type $(i,j,k)$ in $\Gamma$. We often consider whether some distance $k$ occurs in $\Gamma_i$, and this is equivalent to saying that the triangle type $(i,i,k)$ is realized in $\Gamma$.

We refer to $(i,j,k)$ as a \textit{triple} rather than a triangle type if there is doubt as to whether this triple satisfies the triangle inequality.

We will make frequent use of the following fact.
\begin{introfact}{\normalfont \ref{Fact:ProveMHG}}{\bf.}
Let $\Gamma$ be a metrically homogeneous graph and $\sigma$
a permutation of the language. Then the following are equivalent.
\begin{itemize}
\item $\Gamma^\sigma$ is a metrically homogeneous graph.
\item $\Gamma^\sigma$ is a metric space containing all triangles of
type $(1,k,k+1)$ for $k$ less than the diameter of $\Gamma$.
\end{itemize}
\end{introfact}
For the derivation of this fact, see the discussion around Fact \ref{Fact:CharacterizeMHG} in the \hyperref[Fact:CharacterizeMHG]{Appendix}.

We also often exploit the following observation.
\begin{obs}\label{fact:geodesics}
Since by definition any metrically homogeneous graph is connected, every geodesic triangle type $(i,j,i+j)$ where $i+j \leq \delta$ is realized in $\Gamma$.
\end{obs}




\section{Twists} \label{sec:twists}

The following lemmas will be widely applicable in our argument.

\begin{lemma} \label{thm:inverseimageof1}
Let $\sigma$ be a permutation of the language for which there is some metrically homogeneous graph $\Gamma$ twistable by $\sigma$. Let $k = \sigma^{-1}(1)$. Then $\sigma(ik)=i$ for all $i$ satisfying $ik \leq \delta$.
\end{lemma}

\begin{proof}
We proceed by induction. Our base case $i = 1$ holds by definition, so we fix some $i$ with $1 < i \leq \delta/k$ and assume for all $j < i$ that $\sigma(jk) = j$. A geodesic of type $(k,(i-1)k,ik)$ must be realized in $\Gamma$ and therefore its image under $\sigma$, the triangle type $(1,i-1,\sigma(ik))$, must be realized in $\Gamma^{\sigma}$. By the triangle inequality we have $|\sigma(ik) - (i-1)| \leq 1$. As $\sigma(jk) = j$ for all $j < i$, the only remaining option then is $\sigma(ik) = i$.
\end{proof}

Note that from this we may deduce that if $\sigma(1) = 1$, then $\sigma(i) = i$ for all $i \leq \delta$. This point also follows from Fact \ref{Fact:CharacterizeMHG} as the metric $d$ in $\Gamma^{\sigma}$ must be the graph metric.
Thus if $\sigma$ is a non-trivial permutation of the language for which there is some metrically homogeneous graph $\Gamma$ twistable by $\sigma$, then $\sigma(1)> 1$.

We also deduce that the diameter of $\Gamma$ must be finite:

\begin{lemma}\label{lemma:diamfinite}
Let $\Gamma$ be a metrically homogeneous graph and let $\sigma$ be a non-trivial permutation of the language of a metrically homogeneous graph $\Gamma$ 
such that $\Gamma^{\sigma}$ is also a metrically homogeneous graph. Then the diameter $\delta$ of $\Gamma$ is finite.
\end{lemma}

\begin{proof}
Suppose $\delta = \infty$ and let $k = \sigma^{-1}(1)$. By Lemma \ref{thm:inverseimageof1}, we have
\begin{equation*}
    \sigma[k \mathbb{N}] = \mathbb{N}
\end{equation*}
and hence $k \mathbb{N} = \mathbb{N}$. Thus $k = 1$, but then $\sigma$ is trivial.
\end{proof}

\subsection{Non-generic type}

Here we work towards proving the following main result:

\begin{prop}
\label{Prop:Twist:Nongeneric}
Let $\Gamma$ be a metrically homogeneous graph of non-generic type
or diameter $\delta\leq 2$,
and $\sigma$ a non-trivial permutation of the language such that 
$\Gamma^\sigma$ is metrically homogeneous.
Then $\Gamma^\sigma$ is also of non-generic type or diameter $\delta\leq 2$, and one of the following applies.
\begin{itemize}
\item $\Gamma$ has diameter 2 and is not complete multipartite, with $\sigma$ the transposition $(1,2)$.
\item $\Gamma$  is finite, antipodal of diameter $3$, not bipartite, 
with $\sigma$ the transposition $(1,2)$.
\item $\Gamma$ is an $n$-cycle $C_n$, $n\geq 7$; $\delta=\lfloor n/2\rfloor$: 
$\sigma$ is given by multiplication by $\pm k \pmod n$ for $(k,n)=1$, 
with $\pm i$ identified.
\end{itemize}
\end{prop}

We prove this result by first showing the following three claims.


\begin{claimprop}\label{claim:mhgsfordiameter2} For $\delta \leq 2$ the metrically homogeneous graphs with (non-trivial) twists are the homogeneous graphs which are neither disjoint unions of complete graphs nor complements of disjoint unions of complete graphs.
\end{claimprop}

\begin{claimprop}\label{claim:cycles}
Let $\Gamma$ be an $n$-cycle with $n \geq 3$. Then the permutations $\sigma$ of the language of $\Gamma$ for which $\Gamma^{\sigma}$ is a metrically homogeneous graph are given by the group $U_n/(\pm 1)$ acting on the set of distances, where $U_n$ is the group of units modulo $n$. More specifically, we define the action $\mu_k$ on the set of distances to be multiplication by $k$ modulo $n$ where $k \in U_n$. 
\end{claimprop}

\begin{claimprop}\label{claim:diam3nonbipartite}
Let $\Gamma$ be a metrically homogeneous antipodal graph of diameter 3. If $\Gamma$ is not bipartite, then $\sigma = (12)$ is the unique twist for which $\Gamma^{\sigma}$ is metrically homogeneous. If $\Gamma$ is bipartite, then no such permutation exists.
\end{claimprop}

\begin{claimpropproof}
The metrically homogeneous graphs of diameter at most $2$ are precisely the connected homogeneous graphs.
For diameter $\delta =2$, twisting via $\sigma = (12)$ is equivalent to taking the complement of the graph. Since the complement of a homogeneous graph is a homogeneous graph, in order to ensure that both $\Gamma$ and $\Gamma^{\sigma}$ are metrically homogeneous graphs, we would need to show that both $\Gamma$ and $\Gamma^{\sigma}$ are connected homogeneous graphs. 

The only disconnected homogeneous graphs of diameter at most $2$ are those which are the disjoint unions of complete graphs (see \cite{LaW-HG}). 
Thus, the twistable metrically homogeneous graphs of diameter $\delta \leq 2$ are the homogeneous graphs which are neither disjoints unions of complete graphs nor are they the complements of disjoint unions of complete graphs.
\end{claimpropproof}

\begin{claimpropproof}
If $k \in U_n$, then $\mu_k$ is an automorphism of the group $\mathbb{Z}/n\mathbb{Z}$. The group $\mathbb{Z}/n\mathbb{Z}$ can be endowed with the graph structure of $C_n$ by defining two elements $a,b$ to be adjacent if $a-b = \pm 1$. The action of $\mu_k$ on $\mathbb{Z}/n\mathbb{Z}$ yields another graph where now there is a edge between $a$ and $b$ if $a-b = \pm k$. This graph is isomorphic to $C_n$ since $k$ is a generator of $\mathbb{Z}/n\mathbb{Z}$.
Hence the action of $\mu_k$ not only sends $C_n$ to some metrically homogeneous graph, rather it sends $C_n$ to $C_n$.

Now suppose $C_n^{\sigma}$ is a metrically homogeneous graph and $k = \sigma^{-1}(1)$. As $\Gamma$ is connected, all points of $\Gamma$ are connected by paths with successive distances equal to $k$, and thus all distances occurring in $\Gamma^{\sigma}$ are divisible by $\gcd(k,n)$. But one of these distances is $1$, so $k \in U_n^*$. Recalling from Lemma \ref{thm:inverseimageof1} that $\sigma(1)$ determines $\sigma$, we have $\sigma = \mu_k$.
\end{claimpropproof}

\begin{claimpropproof}
Suppose first that $\Gamma$ is antipodal of diameter $3$ and $\Gamma^{\sigma}$ is metrically homogeneous. If $k = \sigma(3)$ then the relation $d(x,y) = k$ defines a pairing on the metrically homogeneous graph $\Gamma^{\sigma}$. That is, for every vertex $v \in \Gamma^{\sigma}$ there exists a unique $v' \in \Gamma^{\sigma}$ such that $d(v,v') = k$. If $k=1$, we would get that $\Gamma^{\sigma}$ is disconnected, as $\delta > 1$. 

Suppose $k=2$. If $\Gamma$ is of generic type then by definition $\Gamma$ contains an infinite set of elements which are pairwise at distance $2$. Thus the relation $d(x,y) = 2$ is not a pairing, i.e. for each vertex $v$ there will be more than one vertex $v'$ such that $d(v,v') = 2$. This violates the antipodality of $\Gamma^{\sigma}$. Hence if $k = 2$ then $\Gamma$ is of non-generic type.

Inspection of the classification as given in Fact \ref{Fact:NongenericType}
of the Appendix then forces $\Gamma$ to be complete multi-partite with classes of size $2$. This then implies that $\delta =2$, which is a contradiction.

So $\sigma(3) = 3$ and if $\sigma$ is non-trivial then it must be the transposition $(12)$. 

But as $\Gamma^{\sigma}$ is connected, and its edge relation is given by $d(x,y) = 2$ in $\Gamma$, it follows that $\Gamma$ cannot be bipartite.

To complete the proof of the claim it suffices to show conversely that when $\Gamma$ is antipodal of diameter $3$ and is not bipartite, and $\sigma$ is the transposition $(12)$, then $\Gamma^{\sigma}$ is a metrically homogeneous graph. By Fact \ref{Fact:ProveMHG} it suffices to show that $\Gamma^{\sigma}$ satisfies the triangle inequality and is connected.
The graph $\Gamma^{-1}$ realizes the geodesic triangle type $(1,2,3) = (1,2,3)^{\sigma^{-1}}$ as it is the inverse image of geodesic. It also realizes the geodesic $(1,1,2) = (1,2,2)^{\sigma^{-1}}$,
since for antipodal graphs of diameter $3$, $K_2 = 2$. 

We translate the remaining conditions on $\Gamma^{\sigma}$ into the language of $\Gamma$. The connectivity  of $\Gamma^{\sigma}$ means that $\Gamma$ is connected with respect to the relation $d(x,y) = 2$, and the triangle inequality means that the triangle type $(\sigma^{-1}(1),\sigma^{-1}(1),\sigma^{-1}(3)) = (2,2,3)$ is not realized in $\Gamma$. 

By Fact \ref{fact:antipodalSmallC}, antipodal graphs contain no triangles of perimeter greater than $2\delta$. Thus $\Gamma$ does not realize the triangle type $(2,2,3)$. This disposes of the second point.

Now consider a connected component $\Omega$ of $\Gamma$ with respect to the edge relation $d(x,y) = 2$. For any distance $k$ which is realized in $\Omega$, by homogeneity we find that any pair of points at distance $k$ lie in the same connected component with respect to this relation. In particular, if the distance $1$ is realized in $\Omega$, then $\Omega = \Gamma$. As $\Gamma$ is not bipartite, the only other possibility is that $\Omega$ realizes the distances $2$ and $3$, and is a connected component of $\Gamma$ with respect to the relation 
\begin{equation*}
    d(x,y) \in \{2,3\}.
\end{equation*}
If $\Omega \neq \Gamma$, then take $u,v \in \Omega$ at distance $3$ and $w \not\in \Omega$ to find ${d(u,w) = d(v,w) = 1}$, contradicting the triangle inequality.

This concludes the proof of the claim.
\end{claimpropproof}

\begin{proof}[Proof of Proposition \ref{Prop:Twist:Nongeneric} from Claims \ref{claim:mhgsfordiameter2}, \ref{claim:cycles}, and \ref{claim:diam3nonbipartite}]
The case of diameter at most $2$ is covered by Claim \ref{claim:mhgsfordiameter2}, bearing in mind that $\Gamma^{\sigma}$ has the same diameter as $\Gamma$. So we suppose $\delta \geq 3$ and $\Gamma$ is of non-generic type. By the classification of non-generic type, given as Fact \ref{Fact:NongenericType} in the Appendix (\S \ref{sec:app}), $\Gamma$ is then finite or one of the tree-like graphs $T_{m,n}$. When $\Gamma$ is finite, Fact \ref{Fact:Finite:delta>=3} describes the possibilities, and Claims \ref{claim:cycles} and \ref{claim:diam3nonbipartite} deal with those possibilities. Of course in this case $\Gamma^{\sigma}$ is also finite and thus not of generic type. Finally, the case of $T_{m,n}$ does not arise since Lemma \ref{lemma:diamfinite} tells us that the diameter must be finite.
\end{proof}


\subsection{Generic type}\label{sec:nec:generic}
We work towards the following result.
\begin{thm}\label{thm:twists}
Let $\sigma$ be a non-trivial permutation of the language of a metrically homogeneous graph $\Gamma$ of generic type where $\Gamma^{\sigma}$ is itself a metrically homogeneous graph. Then $\sigma$ is one of $\rho, \rho^{-1}, \tau_0, \tau_1$, which are defined as follows:

\begin{align*}
\rho(i)&=
\begin{cases} 2i & \mbox{$i\leq \delta/2$}\\
2(\delta-i)+1&\mbox{$i>\delta/2$}
\end{cases}&
\rho^{-1}(i)&=
\begin{cases} i/2 & \mbox{$i$ even}\\
\delta-\frac{i-1}{2}&\mbox{$i$ odd}
\end{cases}
\end{align*}
and for $\epsilon=0$ or $1$, $\tau_{\epsilon}$ is the involution defined by
\begin{align*}
\tau_\epsilon(i)&=
\begin{cases}
(\delta+\epsilon)-i&\mbox{for $\min(i,(\delta+\epsilon)-i)$ odd}\\
i &\mbox{otherwise.}
\end{cases}
\end{align*}
\end{thm}
We begin by showing the following:
\begin{prop} \label{thm:optionsfor1}
Let $\sigma$ be a non-trivial permutation of the language of a metrically homogeneous graph $\Gamma$ of generic type which maps $\Gamma$ to another metrically homogeneous graph.

Then $\delta$ is finite, and $\sigma(1) \in \{2, \delta-1,\delta\}$.
\end{prop}

\begin{proof}
Let $\sigma(1) = k$, and assume towards a contradiction that $2 < k < \delta -1$. By Proposition \ref{Prop:Realize:i,i,2k} of \S \ref{sec:app}, the triangle types $(k,k,2)$ and $(k,k,4)$ must be realized in $\Gamma^{\sigma}$. Thus their inverse images under $\sigma$, namely $(1,1,\sigma^{-1}(2))$ and $(1,1,\sigma^{-1}(4))$, satisfy the triangle inequality. This implies that $\sigma^{-1}(\{2,4\}) = \{1,2\}$. 
Hence
\begin{align*}
    \sigma(1) = 4 & & \sigma(2) = 2 && \delta \geq 6,
\end{align*}
as $\sigma(1) \leq \delta-1$.

We will now argue that this implies that $\delta = 6$, and a contradiction will follow. 

If $\delta \geq 7$, then again by Proposition \ref{Prop:Realize:i,i,2k}, the triangle types $(k,k,2), (k,k,4)$ and $(k,k,6)$ must all be realized in $\Gamma^{\sigma}$. However, there are only two possible values $i$ for which the triple $(1,1,i)$ will satisfy the triangle inequality. We therefore have a contradiction in this case.


Now suppose that $\delta = 6$. Since the triangle type $(2,4,6)$ is of geodesic type,
it must be realized in $\Gamma^{\sigma}$ (Observation \ref{fact:geodesics}), and therefore $\sigma^{-1}(2,4,6) = (1,2,\sigma^{-1}(6))$ must be realized in $\Gamma$. This implies then that $\sigma^{-1}(6) \leq 3$. The only option then is that 
\begin{align*}
\sigma(3) = 6.
\end{align*}

This leaves $\sigma(4) \in \{1,3,5\}$. The geodesic types $(1,3,4)$ and $(2,2,4)$ are realized in $\Gamma$, and therefore their images $(4,6,\sigma(4))$ and $(2,2,\sigma(4))$ are realized in $\Gamma^{\sigma}$. This implies that $2 \leq \sigma(4) \leq 4$. Thus, 
\begin{align*}
    \sigma(4) = 3.
\end{align*}

Finally, we examine $\sigma^{-1}(1) \in \{5,6\}$. The geodesic type $(1,3,4)$ being realized in $\Gamma^{\sigma}$ implies that $(1,4,\sigma^{-1}(1))$ is realized in $\Gamma$, and therefore $\sigma^{-1}(1) \leq 5$. This gives us that $\sigma(5) = 1$, and thus $\sigma = (14365)$. However, this permutation would send the geodesic type $(2,3,5)$ to $(1,2,6)$, and hence is not a suitable twist.

We have therefore indeed shown that $\sigma(1) \in \{2,\delta-1,\delta\}$.
\end{proof}

\begin{prop} \label{thm:rho}
Let $\sigma$ be a permutation of the language of a metrically homogeneous graph $\Gamma$ of generic type which sends $\Gamma$ to another metrically homogeneous graph, satisfying $\sigma(1) = 2$. Then either $\sigma = \rho$ or $\delta = 3$ and $\sigma$ is the transposition $(12)$.

\end{prop}

\begin{proof}
By Lemma \ref{thm:inverseimageof1}, $\sigma(1) = 2i$ for all $i \leq \delta/2$. So the image of $\sigma$ on $(\delta/2,\delta]$ is the set $I$ of odd numbers in the interval $[1,\delta]$.

Consider the geodesic type $(1,i,i+1)$ for any $i < \delta$, which $\sigma$ maps to $(2,\sigma(i),\sigma(i+1))$. Since this triple must satisfy the triangle inequality, we know that 
\begin{equation} \label{eq:neighbors}
|\sigma(i) - \sigma(i+1)| \leq 2.    
\end{equation}
 If $i > \delta /2$, then both $\sigma(i)$ and $\sigma(i+1)$ are odd, and therefore ${|\sigma(i) - \sigma(i+1)| = 2}$. Thus the values $\sigma(i)$ for $i > \delta/2$ give either an increasing or a decreasing enumeration of $I$. In the latter case $\sigma = \rho$. So we suppose that $\sigma(i)$ enumerates $I$ in increasing order for $i > \delta/2$. 
 
 Let $k = \lfloor \delta/2 \rfloor +1$. Then we have in particular that $\sigma(k) =1$. The image of the geodesic triangle type $(1,k-1,k)$ under $\sigma$ is $(2,2k-2,1)$ and the triangle inequality gives $2\lfloor \delta/2 \rfloor \leq 3$, hence $\delta \leq 3$. But then, as $\sigma(1) = 2$, we either have $\sigma = \rho$ or $\sigma = (12)$ with $\delta =3$.
 \end{proof}



\begin{cor*}
Let $\sigma$ be a permutation of the language of a metrically homogeneous graph $\Gamma$ of generic type for which $\Gamma^{\sigma}$ is a metrically homogeneous graph, with $\sigma(2) = 1$. Then either $\sigma = \rho^{-1}$, or $\sigma$ is the transposition $(12)$ and $\delta =3$.
\end{cor*}

\begin{proof}
The graph $\Gamma^{\sigma}$ is twistable by $\sigma^{-1}$ and is of non-generic type by Proposition \ref{Prop:Twist:Nongeneric}. So Proposition \ref{thm:rho} applies to $\sigma^{-1}$ and $\Gamma^{\sigma}$, giving the result.
\end{proof}




We finally show the following:
\begin{prop} \label{thm:taus}
Let $\sigma$ be a permutation of the language of a metrically homogeneous graph $\Gamma$ of generic type with diameter $\delta \geq 3$ such that $\Gamma^{\sigma}$ is itself a metrically homogeneous graph. Assume in addition that $\sigma^{-1}(1) > 2$ and that $\sigma(1) \geq \delta-1$ and $\sigma(1) > 2$. Then either $\sigma = \tau_{\epsilon}$, with $\epsilon = \sigma(1) - (\delta-1) \in \{0,1\}$, or $\sigma = \rho^{-1}$ and $\delta = 3$.
\end{prop}

The proof will be inductive. The base of the induction depends in part on the following.
\begin{lemma} \label{thm:imageof2}
Let $\sigma$ be a permutation of the language of a metrically homogeneous graph $\Gamma$ of generic type with diameter $\delta \geq 3$ such that $\Gamma^{\sigma}$ is itself a metrically homogeneous graph. Assume moreover that $\sigma(1) \geq \delta -1$, $\sigma(1) > 2$, and $\sigma^{-1}(1) > 2$. Then $\sigma(2) = 2$.
\end{lemma}

\begin{proof}
Suppose first that $\sigma(1) = \delta-1$. Since by Proposition \ref{Prop:Realize:i,i,2k} the triangle type $(\delta-1,\delta-1,2)$ is realized in $\Gamma^{\sigma}$, its inverse image $(1,1,\sigma^{-1}(2))$ must satisfy the triangle inequality, meaning that $\sigma^{-1}(2) \leq 2$. Since by assumption $\sigma(1) > 2$, we have that $\sigma(2) = 2$.

Now assume that $\sigma(1) = \delta$. Then $(\Gamma^{\sigma})_{\delta} = (\Gamma_1)^{\sigma}$. There are at most two distances realized in $\Gamma_1$, and hence the same applies to $(\Gamma_1)^{\sigma}$; namely, at most $\sigma(1)$ and $\sigma(2)$ occur. Hence the same applies to $(\Gamma^{\sigma})_{\delta}$; thus the only two distances which may occur in $(\Gamma^{\sigma})_{\delta}$ are $\sigma(1) = \delta$ and $\sigma(2)$. 

Using Fact \ref{Fact:Neighbors:i+-1}, we know that each vertex in $\Gamma_{\delta-1}$ has two neighbors in $\Gamma_{\delta}$. The distance $i$ between these two points is either $1$ or $2$. 
So $i \neq \sigma(1)$. 
We therefore have that $\sigma(2) = i \leq 2$.  Since $\sigma^{-1}(1) > 2$, we have our desired result: $\sigma(2) = 2$.
\end{proof}

\begin{lemma} \label{thm:inversetwoaway}
Let $\sigma$ be a permutation of the language of a metrically homogeneous graph $\Gamma$ of generic type such that $\Gamma^{\sigma}$ is metrically homogeneous and suppose that $\sigma(2) = 2$. Then for $3 \leq k \leq \delta$, we have that
\begin{equation*}
|\sigma(k) - \sigma(k-2)| \leq 2.
\end{equation*}
and 
\begin{equation*}
|\sigma^{-1}(k) - \sigma^{-1}(k-2)| \leq 2.
\end{equation*}
\end{lemma}

\begin{proof}
Apply the triangle inequality to the image under $\sigma$ or $\sigma^{-1}$ of the geodesic type $(2,k-2,k)$. Our assumption that $\sigma(2) = 2$ then yields our desired result.
\end{proof}

We now proceed with the proof of Proposition \ref{thm:taus}.

\begin{proof}[Proof of Proposition \ref{thm:taus}]
We initially assert the following claim.

\begin{claimprop} \label{firstclaimforkeven}
Let $k$ be even and at most $\delta$. Assume moreover that $k \leq (\delta+\epsilon)/2$ or $(\delta + \epsilon)$ is even. Then
\begin{align*}
    \sigma^{-1}(\tau_{\epsilon}(k)) \leq k & & \sigma^{-1}(\tau_{\epsilon}(k-1)) \leq k-1
\end{align*}
\end{claimprop}

\begin{claimpropproof}

Note that for the values specified, $\tau_{\epsilon}(k) = k$. Thus we show that $\sigma^{-1}(k) \leq k$.

We proceed by induction.

For $k=2$, we have from Lemma \ref{thm:imageof2} that $\sigma(2) = 2$. Moreover, by assumption, $\sigma^{-1}(\delta + \epsilon - 1) = 1$. Thus our base case holds. We assume then that $k>2$ and for all even $j < k$ that $\sigma^{-1}(j) \leq j$ and $\sigma^{-1}(\tau_{\epsilon}(j-1)) \leq j-1$. 

By Lemma \ref{thm:inversetwoaway}, we know that $|\sigma^{-1}(k) - \sigma^{-1}(k-2)| \leq 2$. Since by assumption $\sigma^{-1}(k-2) \leq k-2$, we have that $\sigma^{-1}(k) \leq k$.

We consider now $\sigma^{-1}(\tau_{\epsilon}(k-1))$. Note that for even $j$ with $j \leq (\delta+\epsilon)/2$ or ${(\delta+\epsilon)}$ even, we have $\tau_{\epsilon}(j+1)= \tau_{\epsilon}(j-1)-2$, and thus ${\sigma^{-1}(\tau_{\epsilon}(k-1)) = \sigma^{-1}(\tau_{\epsilon}(k-3)-2)}$. Again using Lemma \ref{thm:inversetwoaway}, we get that 
\begin{equation*}
    | \sigma^{-1}(\tau_{\epsilon}(k-1)) - \sigma^{-1}(\tau_{\epsilon}(k-3))|  \leq 2.
\end{equation*}
Since by induction $\sigma^{-1}(\tau_{\epsilon}(k-3)) \leq k-3$, we indeed have that ${\sigma^{-1}(\tau_{\epsilon}(k-1)) \leq k-1}$.
\end{claimpropproof}

We now move on to the next claim:
\begin{claimprop} \label{secondclaimforkeven}
Suppose that $k$ is even and $2 \leq k \leq \delta$. Assume that $k \leq (\delta+\epsilon)/2$ or $\delta + \epsilon$ is even. Then $\sigma(k) = \tau_{\epsilon}(k)$ and $\sigma(k-1) = \tau_{\epsilon}(k-1)$. 
\end{claimprop}

\begin{claimpropproof} 
We begin by noting that for $i \leq (\delta+\epsilon)/2$ or for $\delta + \epsilon$ even, the permutation $\tau_{\epsilon}$ is as follows.

\begin{align*}
\tau_\epsilon(i)&=
\begin{cases}
(\delta+\epsilon)-i&\mbox{$i$ odd}\\
i &\mbox{$i$ even}
\end{cases}
\end{align*}

We proceed by induction. We know by assumption and by Lemma \ref{thm:imageof2} that $\sigma(2) = 2 = \tau_{\epsilon}(2)$ and $\sigma(1) = \delta + \epsilon -1 = \tau_{\epsilon}(1)$. Thus we assume that $k>2$ and for $j$ even, $j \leq k-2$,
that $\sigma(j) = \tau_{\epsilon}(j) = j$ and $\sigma(j-1) = \tau_{\epsilon}(j-1) = \delta + \epsilon - i$.

For the assumed values of $k$ and $\delta+\epsilon$, we have from Claim \ref{firstclaimforkeven} that ${\sigma^{-1}(\tau_{\epsilon}(k)) \leq k}$ and $\sigma^{-1}(\tau_{\epsilon}(k-1)) \leq k-1$. 

Since $\tau_{\epsilon}(k-1) \neq \tau_{\epsilon}(i)$ for any $i < k-1$, it is also the case that $\sigma^{-1}(\tau_{\epsilon}(k-1)) \neq \sigma^{-1}(\tau_{\epsilon}(i))$ for any $i < k-1$. Thus by our inductive hypothesis, $\sigma^{-1}(\tau_{\epsilon})(k-1)$ cannot equal any of $\{ \sigma^{-1}(\sigma(1)), \sigma^{-1}(\sigma(2)),... \sigma^{-1}(\sigma(k-2))\} = \{1,2,...,k-2\}$. Therefore the only possible value that remains is $\sigma^{-1}(\tau_{\epsilon}(k-1)) = k-1$. By a similar argument, we obtain that $\sigma^{-1}(\tau_{\epsilon}(k)) = k$.
\end{claimpropproof}

We note here that one possible value of $\sigma(i)$ has not been explicitly determined from Claim \ref{secondclaimforkeven} when $\delta+\epsilon$ is even, namely $\sigma(\delta)$ when $\delta$ is odd. Of course this is easily resolved. Claim \ref{secondclaimforkeven} tells us that for $\delta+\epsilon$ even, $\sigma(i) = \tau_{\epsilon}(i)$ for all $i < \delta$. Hence $\sigma(\delta) = \tau_{\epsilon}(\delta)$.

It remains to consider the case when
\begin{align*}
    \delta + \epsilon \text{ is odd.}
\end{align*}
We maintain this assumption for the rest of the proof of the proposition.

\begin{claimprop}
$\sigma = \tau_{\epsilon}$.
\end{claimprop}

\begin{claimpropproof}
By assumption $\sigma^{-1}(1) > 2$, and therefore by Proposition \ref{thm:optionsfor1}, ${\sigma^{-1}(1) \in \{\delta -1 , \delta\}}$. We write then $\sigma^{-1}(1) = \delta + \epsilon' -1$ with $\epsilon' \in \{0,1\}$. If $\epsilon' \neq \epsilon$, then $\delta + \epsilon'$ is even, and by from Claim \ref{secondclaimforkeven}, we have that $\sigma^{-1} = \tau_{\epsilon'}$. Since $\tau_{\epsilon'}^{-1} = \tau_{\epsilon'}$, we would have that $\sigma = \tau_{\epsilon'}$, and we would therefore still obtain that $\epsilon' = \epsilon$.
\end{claimpropproof}


Hence we may apply Claims \ref{firstclaimforkeven} and \ref{secondclaimforkeven} to both $\sigma$ and $\sigma'$, yielding the following for $k \leq \delta/2$

\begin{align*}
    \sigma(k) = k  &\mbox{    if $k$ is even}\\
    \sigma(k) = \delta + \epsilon - k  &\mbox{    if $k$ is odd}\\
    \sigma(\delta + \epsilon -k) = k &\mbox{   if $k$ is odd}
\end{align*}

As $\delta+\epsilon$ is odd, it remains to determine $\sigma$ on the set $$A = \{i \hspace{1mm} | \hspace{1mm} (\delta+\epsilon)/2 < i \leq \delta \text{ and $i$ is odd}\}.$$

We already know that $\sigma[A] = A$. Moreover, since all the elements in $A$ are odd, we deduce from Lemma \ref{thm:inversetwoaway} that $\sigma$ either fixes or reverses $A$.

If $\sigma$ fixes $A$, then $\sigma = \tau_{\epsilon}$. Thus we assume towards a contradiction that $\sigma$ reverses $A$ and that $|A| \geq 2$, so $\delta \geq 5$.

We consider first the case when $\epsilon = 1$. Since we are also assuming that $\delta + \epsilon$ is odd, we have that $\delta$ is even and $\max A = \delta - 1$. Under our assumptions then, $\sigma$ would send the geodesic type $(1,\delta-1,\delta)$ to $(1,\min A, \delta)$, which must therefore satisfy the triangle inequality. Therefore $\min A \geq \delta -1$ which would imply that $|A| =1$, which is a contradiction.

Now suppose that $\epsilon = 0$. Then $\delta$ is odd and $\max A = \delta$.
Under these assumptions, $\sigma$ maps the geodesic type $(2, \min A -2, \min A)$ to the triangle type $(2,\delta - \min A + 2, \delta)$. Once again we obtain from the triangle inequality a restriction:
\begin{equation*}
    \delta \leq \delta - \min A + 4.
\end{equation*}
Since $\min A$ is odd, we find that $\min A \leq 3$. However $\min A > \delta /2$, so then $\delta \leq 5$. 

Thus $\delta = 5$, and $\sigma = \sigma^{-1} = (14)(35)(2)$. This permutation sends the geodesic type $(2,2,4)$ to the triangle type $(2,2,1)$ and the forbidden triple $(1,1,4)$ to the triple $(1,4,4)$. Therefore the triangle type $(2,2,1)$ is realized in $\Gamma$ and the triangle type $(4,4,1)$ is not realized in $\Gamma$. So $K_1 \leq 2$ and $\Gamma_4$ contains no edge. By Fact \ref{Fact:LocalAnalysis:K1le2} we find that $\Gamma$ is antipodal. Thus $\Gamma$ does not realize the triangle type $(5,5,2)$ and hence $\Gamma^{\sigma}$ does not realize the triangle type $(3,3,2)$. This contradicts Proposition \ref{Prop:Realize:i,i,2k} of the Appendix.
\end{proof}

\section{Twistable graphs} \label{sec:twistables}

We work towards showing the following main result:

\begin{thm}\label{Thm:Twistable:Graphs}
Let $\sigma$ be one of the permutations $\rho,\rho^{-1}, \tau_0$, or $\tau_1$, with $\delta\geq 3$. Then the metrically homogeneous graphs $\Gamma$ of generic type whose images $\Gamma^{\sigma}$ are also metrically homogeneous are precisely those with the numerical parameters $K_1, K_2, C, C'$ as in Table \ref{Table:Twistable} below.

\begin{table}[ht]
\begin{tabular}{c|ccccc} 
$\sigma$&$\delta$&$K_1$&$K_2$&$C$&$C'$\\
\hline
$\rho$&$\geq 3$, $<\infty$
&$1$&$\delta$&$2\delta+2$&$2\delta+3$\\
$\rho^{-1}$&$\geq 3$, $<\infty$
&$\delta$&$\delta$&$3\delta+1$&$3\delta+2$\\
$\tau_\epsilon$&$\geq 3$, $<\infty$
&$\lfloor \frac{\delta+\epsilon}{2}\rfloor$&$\lceil \frac{\delta+\epsilon}{2}\rceil$&$2(\delta+\epsilon)+1$&$2(\delta+\epsilon)+2$\\
$\tau_\epsilon$&
$\geq 3$, $\equiv \epsilon \pmod 2$&
$\infty$&$0$&$2\delta+1$&$2(\delta+\epsilon)+2$\\
\hline
&\multicolumn{5}{c}{\it Exceptional Cases }\\
&\multicolumn{5}{c} {$\sigma=\tau_1$, $\delta=3$ or $4$}\\
\cline{3-5}
$\tau_1$&$3$&$1$&$2$&$10$&$11$\\
&$3$&$1$&$2$&$9$&$10$\\
&
$3$ &$2$&$2$&$10$&$11$\\
&$4$&$1$&$3$&$11$&$14$ \\
&$4$&$1$&$3$&$11$&$12$ \\
&$4$&$2$&$3$&$11$&$14$ \\
\end{tabular}
\smallskip

\caption{Twistable Metrically Homogeneous Graphs}
\label{Table:Twistable}
\end{table}

\end{thm}

Note that we do not assume that $\Gamma$ is of known type. If $\Gamma$ is of known type, then its isomorphism type will be determined by its numerical parameters together with a set of Henson constraints. But by Theorem \ref{Thm:Twistable:Graphs}, twistability depends only on the value of the numerical parameters.

The values of the parameters given in the tables correspond to the realization or omission of certain triangle types which are either of the form $(k,k,1)$ or of some fixed perimeter, and will be proved in that form.

We break up our analysis into two subsections: one addressing the necessity of these parameter values for twistability, and another addressing the sufficiency of these parameter values for twistability.

\subsection{Necessity of the restrictions on the parameters}

In this section we prove the following.

\begin{prop}\label{prop:necessity}
Let $\sigma$ be one of the permutations $\rho,\rho^{-1},\tau_0$ or $\tau_1$, with ${\delta \geq 3}$. Then the metrically homogeneous graphs of generic type whose images $\Gamma^{\sigma}$ are also metrically homogeneous graphs must have numerical parameters among those shown in Table \ref{Table:Twistable}.
\end{prop}

The numerical parameters were defined in Definition \ref{defn:numpar}, Section \ref{sec:basics}.

We consider twists individually in the following order: $\sigma = \rho, \rho^{-1}, \tau_{\epsilon}$.

\begin{lemma}\label{lemma:necessity:rho}
Let $\Gamma$ be a metrically homogeneous graph of generic type such that $\Gamma^{\rho}$ is also a metrically homogeneous graph. Then the associated numerical parameters $K_1,K_2,C,C'$ for $\Gamma$ must be $1,\delta,2\delta+2,2\delta+3$ respectively.
\end{lemma}

\begin{proof}
The triangle types $(2,2,2)$ and $(1,1,2)$ must both be realized in $\Gamma^{\rho}$, by the definition of generic type.
Therefore their inverse images $(1,1,1)$ and $(\delta,\delta,1)$ must both be realized in $\Gamma$. Thus we already know that 
\begin{align*}
    K_1=1, K_2 = \delta,\text{ and $C \geq 2\delta+2.$}
\end{align*}


Consider a distance $k$ realized in $\Gamma_{\delta}$, that is, the triangle type $(\delta,\delta,k)$ is realized in $\Gamma$. Under $\rho$, this is mapped to $(1,1,\rho(k))$, so $\rho(k) \leq 2$. Thus
$k = 1$ or $\delta$. Since $\Gamma_{\delta}$ is connected (Fact \ref{Fact:LocalAnalysis}) and $\delta \geq 3$, the distance $\delta$ cannot occur. Thus, $\Gamma_{\delta}$ has diameter at most $1$. As the distance $1$ occurs in $\Gamma_{\delta}$, the diameter equals $1$. It follows from Lemma \ref{Fact:C:delta'} that $C = 2\delta+2$ and $C' = 2\delta+3$.
\end{proof}

\begin{lemma}\label{lemma:necessity:rhoinv}
Let $\Gamma$ be a metrically homogeneous graph of generic type such that $\Gamma^{\rho^{-1}}$ is also a metrically homogeneous graph. Then the associated numerical parameters $K_1,K_2,C,C'$ for $\Gamma$ must be $\delta,\delta,3\delta+1,3\delta+2$ respectively.
\end{lemma}

\begin{proof}
For ease of notation, we write $\tilde{\Gamma} = \Gamma^{\rho^{-1}}$.  

The graph $\tilde{\Gamma}$ is twistable by $\rho$, and is of generic type by Proposition \ref{Prop:Twist:Nongeneric}, so its numerical parameters are given by Lemma \ref{lemma:necessity:rho}. We denote them by $\tilde{K_1},\tilde{K_2},\tilde{C},\tilde{C}'$.

We begin by showing the following.

\begin{claimlemma}
$K_1 = K_2 = \delta.$ 
\end{claimlemma}

\begin{claimlemmaproof}
By Corollary \ref{Fact:realizeoddp}, in order to prove the claim, it suffices to find triangle types of all odd perimeters less than $2\delta+1$ which are not realized in $\Gamma$, as well as a triangle type of perimeter $2\delta+1$ which is realized in $\Gamma$. We first find forbidden triangles of perimeter $2k-1$ for $2 \leq k \leq \delta$.


As $\tilde{C} = 2\delta+2$ and $\tilde{C}'=2\delta+3$ (or by the proof of that fact), we have that $\tilde{\Gamma}_{\delta}$ has diameter $1$. Since the triangle type $(\delta,\delta,\delta)$ is not realized in $\tilde{\Gamma}$, the triangle type $(1,1,1)$ is not realized in $\Gamma$. 


Thus we turn our attention to $k$ satisfying $3 \leq k \leq \delta$ and consider the triangle type $(2,2 \lceil k/2 \rceil -2, 2 \lfloor k/2 \rfloor -1)$. This triangle type has perimeter $2k -1$ and is mapped under $\rho^{-1}$ to $(1,\lceil k/2 \rceil -1, \delta - \lfloor k/2 \rfloor +1)$. This triple violates the triangle inequality, since $\lfloor k/2 \rfloor + \lceil k/2 \rceil = k$ and $k < \delta +1$. Thus the triangle types $(2,2 \lceil k/2 \rceil -2, 2 \lfloor k/2 \rfloor -1)$ must be forbidden in $\Gamma$..

To see that some triangle type of perimeter $2\delta+1$ is indeed realized in $\Gamma$, we argue according to the parity of $\delta$. If $\delta$ is even, then the triangle type $(1,\delta,\delta)$ is the image of $(\delta,\delta/2,\delta/2)$ under $\rho^{-1}$, which is of geodesic type. Thus the triangle type $(1,\delta,\delta)$ is realized in $\Gamma$. If $\delta$ is odd, we consider the triangle type ${(3,\delta-1,\delta-1)}$. This has the image $(\delta-1,\frac{\delta-1}{2},\frac{\delta-1}{2})$ under $\rho^{-1}$, which again is of geodesic type. Thus the triangle type $(3,\delta-1,\delta-1)$ is realized in $\Gamma$.

This proves the claim. In particular, $C > 2\delta+1$.
\end{claimlemmaproof}

We turn our attention now to $C$ and $C'$. Here we use some additional structure theory. Since $K_2 = \delta$, we may apply Lemma \ref{Fact:C:delta'} to get that ${C = {2\delta + \delta' +1}}$ and $C' = C+1$ where $\delta'$ is the diameter of $\Gamma_{\delta}$. It remains to show that $\delta'=\delta$, or in other words that a triangle of type $(\delta,\delta,\delta)$ occurs in $\Gamma$.

The permutation $\rho^{-1}$ takes $(\delta,\delta,\delta)$ to $(\frac{\delta+\epsilon}{2},\frac{\delta+\epsilon}{2},\frac{\delta+\epsilon}{2})$, where $\epsilon$ is the parity of $\delta$.

We first consider the case when $\delta$ is even. 
Proposition \ref{Prop:Realize:i,i,2k} gives us that $(\delta/2,\delta/2,2)$ is realized in $\tilde{\Gamma}$ and therefore the triangle type $(\delta,\delta,4)$ is realized in $\Gamma$. Thus $\delta' = \diam(\Gamma_{\delta}) \geq 4$. Moreover, $\Gamma_{\delta}$ is connected by Fact \ref{Fact:LocalAnalysis}. Thus the distance $2$ is realized in $\Gamma_{\delta}$ and hence the triangle type $(\delta/2,\delta/2,1)$ is realized in $\tilde{\Gamma}$. Therefore we may again apply Fact \ref{Fact:LocalAnalysis} to get that $\tilde{\Gamma}_{\delta/2}$ is connected. The distance $\delta$ occurs in $\tilde{\Gamma}_{\delta/2}$, as seen from the geodesic type $(\delta/2,\delta/2,\delta)$, so the distance $\delta/2$ also occurs in $\tilde{\Gamma}_{\delta/2}$. 
Thus the triangle type $(\delta/2,\delta/2,\delta/2)$ is realized in $\tilde{\Gamma}$ and the triangle type $(\delta,\delta,\delta)$ is realized in $\Gamma$.

We now consider the case when $\delta$ is odd. In this case we need the triangle type $(\frac{\delta+1}{2},\frac{\delta+1}{2},\frac{\delta+1}{2})$ to be realized in $\tilde{\Gamma}$. If we can show that the diameter of $\tilde{\Gamma}_{\frac{\delta+1}{2}}$ is at least $\frac{\delta+1}{2}$, then we may argue as we did for $\delta$ even. 


Let $\epsilon'$ be the parity of $j=\frac{\delta+1}{2}$. We show that $\tilde{\Gamma}$ contains the triangle type $(j,j,j+\epsilon')$. The value $j + \epsilon'$ is even and we claim that $j + (j+\epsilon')/2 \leq \delta$, or equivalently $3j + \epsilon' \leq 2\delta$. This is clearly true for $\delta \geq 5$, and for $\delta = 3$, we have that $j = 2$ and $\epsilon'=0$, and thus this inequality holds for all odd $\delta$. 
Thus Proposition \ref{Prop:Realize:i,i,2k} tells us that $(j,j,j+\epsilon')$ is realized in $\tilde{\Gamma}$, and thus are done with the case when $\frac{\delta+1}{2}$ is even.

We are left then with the case $\delta$ odd and $\frac{\delta+1}{2}$ is odd. Note that this implies that $\delta \geq 5$. As in the case when $\delta$ even, we may deduce that $\tilde{\Gamma}_{\frac{\delta+1}{2}}$ is connected. The distance $\frac{\delta+1}{2}+1$ occurs in $\tilde{\Gamma}_{\frac{\delta+1}{2}}$, since by Proposition \ref{Prop:Realize:i,i,2k} the triangle type $(\frac{\delta+1}{2},\frac{\delta+1}{2},\frac{\delta+1}{2}+1)$ is realized in $\tilde{\Gamma}$. Thus the connectivity of $\tilde{\Gamma}_{\frac{\delta+1}{2}}$ yields that the distance $\frac{\delta+1}{2}$ is also realized in $\tilde{\Gamma}_{\frac{\delta+1}{2}}$ and therefore $(\delta,\delta,\delta)$ is realized is $\Gamma$.

Our claim is now complete.
\end{proof}

Our analysis of the parameter values compatible with a twist of the form $\tau_{\epsilon}$ is somewhat more involved, and involves some exceptional cases, as seen in Table \ref{Table:Twistable}. We break up our analysis into a series of lemmas. 

\begin{lemma}\label{lemma:necessity:taubipartite}
Let $\Gamma$ be a metrically homogeneous graph of generic type and diameter at least $3$ such that $\Gamma^{\tau_{\epsilon}}$ is metrically homogeneous. Suppose moreover that $\Gamma$ is bipartite. Then $\delta \equiv \epsilon \pmod 2$ and $C_0 = 2(\delta + \epsilon) +2$.
\end{lemma}

\begin{proof}
When $\delta \not \equiv \epsilon \pmod 2$, then $\tau_{\epsilon}^{-1} = \tau_{\epsilon}$ maps the geodesic $(\frac{\delta+\epsilon-1}{2}, \frac{\delta+\epsilon-1}{2}, 2 \frac{\delta+\epsilon-1}{2})$ to either $(\frac{\delta+\epsilon-1}{2},\frac{\delta+\epsilon-1}{2},\delta+\epsilon-2 \frac{\delta+\epsilon-1}{2}) $ or $(\delta+\epsilon - \frac{\delta+\epsilon-1}{2}, \delta + \epsilon - \frac{\delta+\epsilon-1}{2},\delta+\epsilon - 2 \frac{\delta+\epsilon-1}{2})$. In either case, this implies that a triangle type of odd perimeter in $\Gamma$ is being mapped to a geodesic in $\Gamma^{\tau_{\epsilon}}$. This is a contradiction, since bipartite graphs have no triangle types of odd perimeter. Thus $\tau_{\epsilon}$ is not a viable twist.


When $\delta \equiv \epsilon \pmod 2$, the parities of the distances between elements of $\Gamma$ are preserved under $\tau_{\epsilon}$. Thus the image $\Gamma^{\tau_{\epsilon}}$ is bipartite, with the same parts as $\Gamma$.

We turn our attention now to $C_0$ and we show first that $C_0 \geq 2(\delta + \epsilon) + 2$. If $\epsilon = 0$, this holds by definition. If $\epsilon = 1$, then $\tau_{\epsilon}$ maps the triangle type $(\delta,\delta,2)$ to $(1,1,2)$ which is a geodesic and therefore is realized in $\Gamma^{\tau_{\epsilon}}$. Thus $C_0 > 2\delta+2$, and therefore $C_0 \geq 2(\delta + \epsilon) + 2$.

Thus in order to show that $C_0 = 2(\delta + \epsilon) +2$, it suffices to prove that any triangle type of diameter $2(\delta + \epsilon) +2$ is forbidden in $\Gamma$. Working towards a contradiction, we assume that such a triangle type is realized in $\Gamma$. Fact \ref{Fact:Realize:delta,delta,d} then says that $\Gamma$ has a triangle of type $(\delta,\delta,d)$ with $d \geq 2\epsilon + 2$.


If $\epsilon = 0$, then we consider a pair of vertices $u,v \in \Gamma_{\delta}$ at distance $d$. Using homogeneity, we take $u',v'$ adjacent to $u$ and $v$ respectively so that ${d(u',v) \geq \min(d+1,\delta-1)}$ and $${d(u',v') \geq \min(d(u',v) +1,\delta-1) \geq \min(d+2,\delta-1)}.$$ Let $d' = d(u',v')$.

As $\Gamma$ is bipartite we have $u',v' \in \Gamma_{\delta-1}$ and thus $u',v'$ and the basepoint form a triangle of type $(\delta-1,\delta-1,d')$. The permutation $\tau_0$ maps this triangle type to $(1,1,\tau_0(d'))$. As $\Gamma^{\tau_0}$ is bipartite we find $\tau_0(d') = 2$ and hence $d' = 2$. But as $\delta \equiv \epsilon \pmod 2$, $\delta$ is even and thus
\begin{equation*}
    d' \geq \min(d+2,\delta-1) \geq 3,
\end{equation*}
a contradiction.

If $\epsilon = 1$, then our assumption and Fact \ref{Fact:Realize:delta,delta,d} would imply that $\Gamma$ contains the triangle type $(\delta,\delta,d)$ for some $d \geq 4$. The permutation $\tau_1$ sends this triangle type to $(1,1,\tau_1(d))$ and as $\Gamma^{\tau_1}$ is bipartite we have $\tau_1(d) = 2$ and $d = 2$, which is a contradiction.

Therefore every triangle type of perimeter $2(\delta + \epsilon) + 2$ is forbidden, and thus $C_0 = 2(\delta + \epsilon) + 2$.
\end{proof}

\begin{lemma} \label{lemma:necessity:taunotbipartite}
Let $\Gamma$ be a metrically homogeneous graph of generic type such that $\Gamma^{\tau_{\epsilon}}$ is metrically homogeneous. Suppose furthermore that $\Gamma$ is not bipartite. Then one of the following holds:

\begin{itemize}
\item The unique distance occurring in $\Gamma_{\delta+\epsilon-1}$ is $2$, and if $\epsilon =0$ then $\delta \geq 4$;
\item $\epsilon=0, \delta = 3$; in this case, we have $\Gamma \simeq \Gamma^3_{1,2,7,8}$ (the generic antipodal graph of diameter $3$);
\item $\Gamma$ is in one of the exceptional cases listed with $\epsilon =1$ and $\delta \leq 4$, and the distance $\delta$ occurs in $\Gamma_{\delta}$.
\end{itemize}

\end{lemma}

\begin{proof}
We prove this result via a series of claims.

\begin{claimlemma}\label{claimlemma:casethree}
If $\delta = 3$ and $\epsilon = 0$ then $\Gamma \simeq \Gamma_{1,2,7,8}^{3}$.
\end{claimlemma}

\begin{claimlemmaproof}
By Fact \ref{Fact:MH3}, any triangle type realized in the canonical metrically homogeneous graph $\Gamma_{K_1,K_2,C,C'}^3$ with the same numerical parameters as $\Gamma$ will also be realized in $\Gamma$.
Thus any forbidden triangle types must be directly excluded by one of the parameters. 
The triangle type $(2,3,3)$ corresponds under $\tau_0$ to the triple $(1,1,3)$, which violates the triangle inequality.
Therefore the type $(2,2,3)$ must be excluded from $\Gamma$, and is hence excluded either by the value of a parameter $K_1$ or $K_2$, or by the value of $C_1$. As $\Gamma$ is not bipartite, the type $(2,2,3)$ is not excluded by $K_1$. 
An equivalent definition of $K_2$ would forbid triangles of type $(i,j,k)$ for which $i + j + k > 2K_2 + 2\min(i,j,k)$ and odd (see, for example, \cite[Definition 1.17]{Che-HOGMH}).
Thus, to be excluded by $K_2$ would mean that $7> 2K_2 + 2 \min(2,2,3) = 2K_2+4$ and hence $K_2 = 1$.
By the definition of $K_1$ and $K_2$ we have $K_1 = 1$ as well, and as $\delta-1 =2$, Fact \ref{Fact:LocalAnalysis:K1le2} gives a contradiction. So this triangle type is not excluded by $K_2$.
Thus the triangle type $(2,2,3)$ can only be excluded by the parameter $C_1$, that is, $C_1 = 7$. Now we may use the classification from \cite{ACM-MH3} to identify $\Gamma$ (cf. Theorem 1). 
\end{claimlemmaproof}

\begin{claimlemma}\label{claim:distance2occurs}
If $\delta+\epsilon > 3$, then the only possible distances which may occur in $\Gamma_{\delta + \epsilon - 1}$ are $2$ and $\delta + \epsilon -1$, and at least the distance $2$ does occur.
\end{claimlemma}

\begin{claimlemmaproof}
Consider a triangle type of the form $(\delta + \epsilon -1, \delta + \epsilon -1, k)$ realized in $\Gamma$ with $1 \leq k \leq \delta$. The permutation $\tau_{\epsilon}$ maps this triangle type to $(1,1,\tau_{\epsilon}(k))$. Thus $\tau_{\epsilon}(k) \leq 2$ and $k$ is either $2$ or $\delta + \epsilon -1$, as claimed. For $k = 2$, as $\delta+\epsilon>3$ we have $(1,1,\tau_{\epsilon}(k)) = (1,1,2)$, which is a geodesic type. So the distance $2$ does occur.
\end{claimlemmaproof}


\begin{claimlemma}\label{claimlemma:distancerealized}
If $\epsilon = 0$ and $\delta \geq 4$, then the distance $\delta -1$ is not realized in $\Gamma_{\delta -1}$.
\end{claimlemma}

\begin{claimlemmaproof}
Suppose that $\epsilon = 0$ and that $\Gamma_{\delta-1}$ does realize the distance $\delta -1$. Then $(\delta-1,\delta-1,\delta-1)^{\tau_0} = (1,1,1)$ is realized in $\Gamma^{\tau_{\epsilon}}$. By Fact \ref{Fact:LocalAnalysis:K1le2},
the triangle type $(\delta -1, \delta-1,1)$ is also realized in $\Gamma^{\tau_0}$. Thus its inverse image $(1,1,\delta-1)$ must be in $\Gamma$, and therefore $\delta \leq 3$, a contradiction.
\end{claimlemmaproof}

\begin{claimlemma}\label{claimlemma:exceptional}
If $\epsilon = 1$ and the distance $\delta$ is realized in $\Gamma_{\delta}$, then $\delta \leq 4$ and $\Gamma$ is one of the listed exceptional cases.
\end{claimlemma}

\begin{claimlemmaproof}
We show first that there is a triangle of type $(2,2,\delta)$ in $\Gamma$, and in particular $\delta \leq 4$.
Note that since $\delta$ is realized in $\Gamma_{\delta}$, we have $K_1=1$ for $\Gamma^{\tau_{1}}$ since $(\delta,\delta,\delta)$ in $\Gamma$ corresponds to $(1,1,1)$ in $\Gamma^{\tau_{1}}$. We may apply Proposition \ref{Prop:Realize:i,i,2k} to get that there is a triangle of type $(2,2,1)$ in $\Gamma^{\tau_1}$, and therefore a triangle of type $(2,2,\delta)$ in $\Gamma$. Therefore $\delta \leq 4$.

\textit{Case 1: $\delta =3$}

We are assuming that the triangle type $(3,3,3)$ is realized in $\Gamma$ and we have also shown that the triangle type $(2,2,3)$ is realized in $\Gamma$, since $\delta = 3$. By Claim \ref{claim:distance2occurs}, the triangle type $(2,3,3)$ is also realized in $\Gamma$, and the triangle type $(1,3,3)$ is not realized in $\Gamma$. In particular, there are triangles of perimeters $7,8,9$ in $\Gamma$ and thus $C = 10$ and $C' = 11$.
Moreover, $K_1 \leq K_2 \leq 2$. Since all the metrically homogeneous graphs of diameter $3$ are known, we refer to the catalog in  \cite{ACM-MH3} to deduce that $K_2 = 2$ and $K_1 = 1$ or $2$. These correspond to the first and third exceptional cases in Table \ref{Table:Twistable}, as claimed.

\textit{Case 2: $\delta = 4$}

In this case, by assumption, the triangle type $(4,4,4)$ is realized in $\Gamma$. Thus, $C_0 = 14$ and $\Gamma$ is not antipodal. The triangle type $(3,4,4)$ would be mapped under $\tau_1$ to $(1,1,3)$ and therefore must be forbidden. As this is the only possible triangle type of perimeter $11$, we see that $C_1 \leq 11$. The triangle type $(2,3,4)$ is mapped to $(1,2,3)$ and thus is realized in $\Gamma$. Since this triangle type has perimeter $9$, we also have that $C_1 = 11$.

The triangle type $(4,4,1)$ is forbidden from being realized in $\Gamma$ since it would be mapped under $\tau_{1}$ to $(1,1,4)$. We see however that the triangle type $(2,2,1)$ is realized in $\Gamma$, since it is mapped to $(2,2,4)$. 
Thus $K_2 < 4$ and $K_1 \leq 2$.

We may therefore apply Fact \ref{Fact:LocalAnalysis:K1le2} to get that $(1,3,3)$ is realized in $\Gamma$, yielding $K_2 = 3$.

Once again, $K_1 = 1$ or $2$ corresponds to the first and third exceptional cases for $\delta = 4$, as claimed.

Thus the claim holds in all cases. 
 \end{claimlemmaproof}
 
We argue now that Lemma \ref{lemma:necessity:taunotbipartite} follows from Claims \ref{claimlemma:casethree}, \ref{claim:distance2occurs}, \ref{claimlemma:distancerealized}, and \ref{claimlemma:exceptional}. Recall from Proposition \ref{Prop:Twist:Nongeneric} that we may assume $\delta \geq 3$.

If $\delta + \epsilon < 4$, then $\delta = 3$ and $\epsilon = 0$, then we arrive at the second case mentioned the lemma with Claim \ref{claimlemma:casethree} providing the additional information about $\Gamma$.

We suppose then that
\begin{align*}
    \delta + \epsilon \geq 4.
\end{align*}

If the distance $\delta + \epsilon -1$ does not occur in $\Gamma_{\delta}$ then by Claim \ref{claim:distance2occurs} the only distance occurring in $\Gamma_{\delta}$ is $2$, as in the first case of our lemma. Thus we finally suppose that
\begin{align*}
    \text{The distance } \delta + \epsilon -1 \text{ occurs in } \Gamma_{\delta}.
\end{align*}
The Claim \ref{claimlemma:distancerealized} shows that $\epsilon = 1$. As we are supposing that the distance $\delta$ is realized in $\Gamma_{\delta}$, Claim \ref{claimlemma:exceptional} shows that we are in one of the corresponding exceptional cases with $\delta \leq 4$.
\end{proof}

\begin{lemma}\label{lemma:necessity:notbipartiteantipodal}
Let $\Gamma$ be a metrically homogeneous graph such that $\Gamma^{\tau_0}$ is metrically homogeneous. Suppose moreover that $\Gamma$ is not bipartite. Then $\Gamma$ is antipodal and $K_1 = \left \lfloor \frac{\delta}{2} \right \rfloor$, $K_2 = \left \lceil \frac{\delta}{2} \right \rceil$.
\end{lemma}

\begin{proof}

Note that for this lemma we are only working with $\epsilon = 0$.

In Lemma \ref{lemma:necessity:taunotbipartite} the third case is excluded, and in the second case our lemma holds since the graph is $\Gamma^3_{1,2,7,8}$.
Thus we restrict ourselves to the remaining case of Lemma \ref{lemma:necessity:taunotbipartite}, where the unique distance realized in $\Gamma_{\delta-1}$ is $2$ and $\delta \geq 4$.

\begin{claimlemma} \label{claim:boundeddistance}
For $u \in \Gamma_{\delta-1}, v \in \Gamma_{\delta}$, we have $d(u,v) \leq \delta -2$.
\end{claimlemma}

\begin{claimlemmaproof}
If $d(u,v) = \delta$, then we take $v' \in \Gamma_{\delta-1}$ adjacent to $v$. That would imply that $d(u,v') \geq \delta-1$, but since both $u$ and $v'$ are in $\Gamma_{\delta-1}$, we have $d(u,v') = 2$. As $\delta \geq 4$, this is a contradiction.

If $d(u,v) = \delta-1$, we find a $v'$ adjacent to $v$ such that $d(u,v') = \delta$. By the previous paragraph, $v' \not \in \Gamma_{\delta}$. But then $v' \in \Gamma_{\delta-1}$, and since the unique distance realized in $\Gamma_{\delta-1}$ is $2$, we would have that $\delta = 2$, a contradiction.
\end{claimlemmaproof}

\begin{claimlemma}  \label{claim:antipodal}
$\Gamma$ is antipodal.
\end{claimlemma}

\begin{claimlemmaproof}
Working towards a contradiction, we assume there are two distinct points $u,v$ in $\Gamma_{\delta}$ and we assume without loss of generality that $\delta' = d(u,v) = \diam(\Gamma_{\delta})$. We notice first that there must be a $u' \in \Gamma_{\delta-1}$ adjacent to $u$ with $d(u',v) \geq \min(\delta'+1,\delta-1)$. Indeed, if $\delta' < \delta$, then by the homogeneity of $\Gamma_{\delta}$, we may take $u'$ to be adjacent to $u$ with $d(u',v) = \delta'+1$. Then $u'$ is in $\Gamma_{\delta-1}$. If on the other hand $\delta' = \delta$, we may simply take any $u'$ in $\Gamma_{\delta-1}$ that is adjacent to $u$.

Take $u_1 \in \Gamma_{\delta-1}, v_1 \in \Gamma_{\delta}$ with $d(u_1,v_1)$ maximal; in particular, $d(u_1,v_1) \geq d(u,v')$. By the previous claim $d(u_1,v_1) < \delta$ and thus there is $v'$ adjacent to $v_1$ with $d(u_1,v') = d(u_1,v_1)+1$. By the choice of $u_1,v_1$, we have $v' \not\in \Gamma_{\delta}$, so $v' \in \Gamma_{\delta-1}$. But
\begin{equation*}
    d(u_1,v') \geq d(u',v) + 1 \geq \min(\delta'+2,\delta) > 2
\end{equation*}
a contradiction.

Thus, $\Gamma$ must be antipodal.
\end{claimlemmaproof}

\begin{claimlemma} \label{claim:kvalues}
$K_1 = \lfloor \frac{\delta}{2} \rfloor$ and $K_2 = \lceil \frac{\delta}{2} \rceil$.
\end{claimlemma}

\begin{claimlemmaproof}
We begin by considering any triangle type of the form $(i,i,1)$ which is realized in $\Gamma$. If $\min(i,\delta-i)$ is even then $\tau_0$ sends this triangle type to $(i,i,\delta-1)$. The triangle inequality and the perimeter bound afforded by antipodality would then yield 
\begin{align*}
\delta -1 \leq 2i \leq \delta + 1. 
\end{align*}
If $\min(i,\delta-i)$ is odd, the corresponding inequalities found would be $$\delta -1 \leq 2(\delta-i) \leq \delta+1$$ which are equivalent to the inequalities in the even case. Thus we have that $K_1 \geq \lfloor \frac{\delta}{2} \rfloor$ and $K_2 \leq \lceil \frac{\delta}{2} \rceil$. Recalling that $K_1 + K_2 = \delta$ for antipodal graphs --- see, for example, {\cite[pg. 13]{ACM-MH3}} ---  our claim is shown.
\end{claimlemmaproof}

Claims \ref{claim:antipodal} and \ref{claim:kvalues} prove the lemma.
\end{proof}

The following will be used in the proof of Lemma \ref{lemma:necessity:nonbipartiteorexceptional}.

\begin{lemma}\label{lemma:exceptionalcases}
Let $\Gamma$ be a metrically homogeneous graph such that $\tilde{\Gamma} = \Gamma^{\tau_1}$ is also metrically homogeneous. Suppose moreover that $\Gamma$ satisfies the following conditions.
\begin{itemize}
\item $\delta = 3$ or $4$
\item The distance $\delta$ is realized in $\Gamma_{\delta}$
\item The numerical parameters associated with $\Gamma$ are those associated with one of the exceptional cases in Table \ref{Table:Twistable}, namely one of the following.
\begin{align*}
K_1 \leq 2 && K_2 = \delta -1 && C = \delta + 7 && C' = 3\delta+2\\
K_1 = 1 && K_2 = \delta -1 && C = 2\delta+3 && C' = C+1
\end{align*}
\end{itemize}

If the unique distance occurring in $\tilde{\Gamma}_{\delta}$ is $2$ then $K_1 = 2$ and $\tilde{\Gamma}$ is also one of the exceptional cases listed, with parameters 
\begin{align*}
\tilde{K_1} = 1 && \tilde{K_2} = \delta-1 && \tilde{C} = 2\delta+3 && \tilde{C'} = 2\delta+4
\end{align*}
\end{lemma}

\begin{proof}
If $K_1 = 1$, then the triangle type $(1,1,1)$ is realized in $\Gamma$, and therefore the triangle type $(\delta,\delta,\delta)$ is realized in $\tilde{\Gamma}$, contradicting the assumption that the unique distance in $\tilde{\Gamma}_{\delta}$ is $2$. Thus, $K_1 = 2$.

In these cases the triangle type $(\delta,\delta,\delta)$ is realized in $\Gamma$, so the type $(1,1,1)$ is realized in $\tilde{\Gamma}$, yielding \begin{equation*}
\tilde{K_1} = 1
\end{equation*}
The distance $\delta$ is $\tilde{\Gamma}$ corresponds to the distance $1$ is $\Gamma$, but the relation $d(x,y) = 1$ is not a pairing on $\Gamma$. That is, given a vertex $v \in \Gamma$, there is more than one vertex $v'$ such that $d(v,v') = 1$. Thus $\tilde{\Gamma}$ is not antipodal, and we may apply Fact to get 
\begin{equation*}
\tilde{K_2} \geq \delta-1
\end{equation*}
The distance $1$ does not occur in $\tilde{\Gamma}_{\delta}$, since the triple $(1,1,\delta) = (\delta,\delta,1)^{\tau_1}$ violates the triangle inequality. Thus
\begin{equation*}
\tilde{K_2} = \delta-1
\end{equation*}
Since by assumption the unique distance in $\tilde{\Gamma}_{\delta}$ is $2$, Lemma \ref{Fact:C:delta'} of the Appendix tells us that $\tilde{C}' = 2\delta + 4$, and $\tilde{C} = 2\delta+3$ or $\tilde{C}= 2\delta+1$.

Thus it remains to show that 
\begin{align*}
\tilde{C} \neq 2\delta+1
\end{align*}
We argue then that a triangle of type $(2,\delta-1,\delta)$ is realized in $\tilde{\Gamma}$. 

The image $(2,\delta-1,\delta)^{\tau_1} = (2,\delta-1,1)$ is a geodesic, and therefore realized in $\Gamma$ if $\delta = 4$. It is of type $(2,2,1)$ if $\delta = 3$. This is realized is $\Gamma$ since $K_2 = 2$; hence $(2,\delta-1,\delta)$ is realized in $\tilde{\Gamma}$, leaving $\tilde{C} = 2\delta+3$.

This concludes the proof.
\end{proof}

\begin{lemma}\label{lemma:necessity:nonbipartiteorexceptional}
Let $\Gamma$ be a metrically homogeneous graph such that $\Gamma^{\tau_1}$ is metrically homogeneous and assume that $\Gamma$ is not bipartite. Then either the parameters $K_1,K_2,C,C'$ have the values $\lfloor \frac{\delta+1}{2} \rfloor, \lceil \frac{\delta+1}{2} \rceil, 2\delta +3, 2\delta +4$ respectively, or $\Gamma$ is in one of the exceptional cases of Table \ref{Table:Twistable} with $\delta \leq 4$.
\end{lemma}

\begin{proof}
If the distance $\delta$ occurs in $\Gamma_{\delta}$, then Lemma \ref{lemma:necessity:taunotbipartite} provides our desired result. So we suppose that the unique distance occurring in $\Gamma_{\delta}$ is $2$.

If the distance $\delta$ occurs in $\Gamma^{\tau_1}$ then by Lemma \ref{lemma:necessity:taunotbipartite} the assumptions of Lemma \ref{lemma:exceptionalcases} are applicable to $\Gamma^{\tau_1}$, yielding that $\Gamma$ is one of the exceptional cases of Table \ref{Table:Twistable}. Thus we also suppose that the unique distance realized in $(\Gamma_{\delta})^{\tau_1}$ is $2$.

Fact \ref{Fact:Realize:delta,delta,d} then implies that no triangle realized in $\Gamma$ or in $\Gamma^{\tau_1}$ can have perimeter greater than $2\delta+2$.


Consider a triangle type of the form $(i,i,1)$ realized in $\Gamma$. The permutation $\tau_1$ sends this triangle type to either $(i,i,\delta)$ or $(\delta-i+1,\delta-i+1,\delta)$. The triangle inequality and the perimeter bounds applied to these two triangle types 
yield the following pairs of inequalities in the two cases, respectively.
\begin{align*}
\delta \leq 2i \leq \delta +2 \\
\delta \leq 2(\delta-i+1) \leq \delta +2.
\end{align*}
Note that these two pairs of inequalities are in fact equivalent and may be written in the form 
\begin{equation*}
\left\lfloor \frac{\delta+1}{2} \right\rfloor \leq i \leq \left\lceil \frac{\delta+1}{2} \right\rceil
\end{equation*}
Given that $\Gamma$ is not bipartite, the triangle type $(i,i,1)$ must be realized for some finite $i$. Thus $\lfloor \frac{\delta+1}{2} \rfloor \leq K_1 \leq K_2 \leq \lceil \frac{\delta+1}{2} \rceil$.

We now consider separately the case when $\delta$ is odd and the case when $\delta$ is even.

If $\delta$ is odd, then the values $K_1$ and $K_2$ are squeezed to be 
\begin{align*}
K_1 = K_2 = \frac{\delta+1}{2}.
\end{align*}
 Since $\tau_1 ^ {-1} = \tau_1$, the same applies for $\Gamma^{\tau_1}$. Thus $(1,\frac{\delta+1}{2},\frac{\delta+1}{2})$ is realized in $\Gamma^{\tau_1}$ and gets mapped under $\tau_1$ to $(\frac{\delta+1}{2},\frac{\delta+1}{2},\delta)$, so $\Gamma$ realizes a triangle type of perimeter $2\delta +1$. 
The graph $\Gamma$ also realizes the triangle type $(\delta,\delta,2)$ of perimeter $2\delta +2$.
By Fact \ref{Fact:Realize:delta,delta,d}, $\Gamma$ contains no triangle of perimeter larger than $2\delta + 2$. Thus, 
\begin{align*}
    C = 2\delta +3 \text{   and   } C' = 2 \delta + 4.
\end{align*}

If $\delta$ is even, then $\tau_1$ maps the geodesic triangle type $(1,\delta/2,\delta/2+1)$ from $\Gamma^{\tau_1}$ to $(\delta/2,\delta/2+1,\delta)$ in $\Gamma$. This triangle type has perimeter $2\delta +1$. Using the same reasoning as in the case of $\delta$ odd, we have again that $C = 2\delta +3$ and $C' = 2\delta +4$.

We address now the parameters $K_1$ and $K_2$. In the even case, the above inequalities are 
\begin{align*}
\delta/2 \leq K_1 \leq K_2 \leq \delta/2 + 1.    
\end{align*}

 The geodesic type $(\delta/2,\delta/2,\delta)$ is mapped to $(1,\delta/2,\delta/2)$ and hence ${K_1 = \delta/2}$.
 
As $\Gamma_{\delta}$ has diameter $2$, Lemma \ref{lemma:diameterofsubgraph} yields that $\diam(\Gamma_{\delta/2+1}) = \delta$. The same may be said for $\Gamma^{\tau_1}$. Thus the triangle type $(\delta/2 + 1, \delta/2+1,\delta)$ is in $\Gamma^{\tau_1}$ and therefore the triangle type $(1,\delta/2+1,\delta/2+1)$ is in $\Gamma$. This tells us that $K_2 = \delta/2+1$.
\end{proof}

\begin{proof}[Proof of Proposition \ref{prop:necessity}]

Lemmas \ref{lemma:necessity:rho} and \ref{lemma:necessity:rhoinv} deal with the cases of $\rho$ and $\rho^{-1}$. Lemma \ref{lemma:necessity:taubipartite} treats $\tau_1$ in the bipartite case. Lemmas \ref{lemma:necessity:notbipartiteantipodal} and \ref{lemma:necessity:nonbipartiteorexceptional} treat the cases of $\tau_0$ and $\tau_1$ respectively, in the non-bipartite case, including the exceptional cases. 
\end{proof}

\subsection{Sufficiency of the restrictions of the parameters}

We aim at the following.

\begin{prop}\label{prop:sufficiency}
Let $\sigma$ be one of the permutations $\rho, \rho^{-1}, \tau_0,$ or  $\tau_1$, with $\delta \geq 3$. Let $\Gamma$ be a metrically homogeneous graph whose numerical parameters are given in Table \ref{Table:Twistable}. Then $\Gamma$ is twistable by the corresponding permutation $\sigma$.
\end{prop}

We begin with the case of $\rho$.

\begin{lemma}\label{lemma:sufficiency:rho}
Let $\Gamma$ be a metrically homogeneous graph of generic type with numerical parameters 
\begin{align*}
K_1 = 1 && K_2 = \delta && C = 2\delta + 2 && C' = 2\delta+3.
\end{align*}
Then $\Gamma^\rho$ is metrically homogeneous.
\end{lemma}

\begin{proof}
Here we make use of Fact \ref{Fact:ProveMHG} of the Appendix, which tells us that in order to show that $\Gamma^{\rho}$ is metrically homogeneous, it suffices to check that the triangle type $(i,j,k)$ is not realized in $\Gamma^{\rho}$ for $i + j < k$, and that the triangle type $(1,k,k+1)$ is realized in $\Gamma^{\rho}$ for $1 \leq k < \delta$. Thus we need to check that the corresponding triangle type $(i,j,k)^{\rho^{-1}}$ is not realized in $\Gamma$ and that $(1,k,k+1)^{\rho^{-1}}$ is realized in $\Gamma$.

\begin{claimlemma} 
For $i+j < k $, the triangle type $(i,j,k)^{\rho^{-1}}$ is not realized in $\Gamma$.
\end{claimlemma}

\begin{claimlemmaproof}
We argue according to the parities of $i,j$ and $k$.

If $i,j$ and $k$ are all even, then the triple $(i,j,k)^{\rho^{-1}}$ violates the triangle inequality, since $i/2 + j/2 < k/2$. If $i$ and $j$ are even but $k$ is odd, then the triple $(i,j,k)$ still violates the triangle inequality. Indeed, then ${i + j \leq k-2}$ and therefore $i + j + k \leq 2k-2 \leq 2\delta -2$. Then we conclude that $${i/2 + j/2 \leq \delta - k/2 -1 \leq \delta - \frac{k-1}{2}}$$. As $(i,j,k)^{\rho^{-1}} = (i/2,j/2,\delta-\frac{k-1}{2})$, this triple violates the triangle inequality.

If $i$ and $j$ are both odd, then we show that the perimeter of $(i,j,k)^{\rho^{-1}}$ is greater than $2\delta+1$, thereby violating the $C,C'$ bounds. Indeed, if $k$ is also odd, then $(i,j,k)^{\rho^{-1}}$ is $(\delta-\frac{i-1}{2},\delta-\frac{j-1}{2},\delta-\frac{k-1}{2})$ and therefore has perimeter $3\delta - \frac{i+j+k-3}{2}$. Since $i + j < k \leq \delta$, we have that $$3\delta - \frac{i+j+k-3}{2} > 3\delta-\frac{2k-3}{2} \geq 2\delta+3/2$$. 
If $k$ is even, then $(i,j,k)^{\rho^{-1}} =(\delta-\frac{i-1}{2},\delta-\frac{j-1}{2},k/2)$ and thus has perimeter $2\delta + \frac{k-i-j}{2} + 1 > 2\delta + 1$.

If $i$ and $j$ have opposing parity, then we show that $(i,j,k)^{\rho^{-1}}$ violates the triangle inequality. We assume without loss of generality that $i$ is odd and $j$ is even.
If $k$ is even, then $(i,j,k)^{\rho^{-1}} = (\delta-\frac{i-1}{2},j/2,k/2)$ and $j/2 + k/2 < \delta - \frac{i-1}{2}$ because $i+j+k < 2k-1 \leq 2\delta-1$.
If $k$ is odd, then ${(i,j,k)^{\rho^{-1}} = (\delta-\frac{i-1}{2},j/2,\delta-\frac{k-1}{2})}$ and $(\delta-\frac{k-1}{2}) + j/2 < \delta - \frac{i-1}{2}$ because $i+j < k$.

Thus none of these triples may be realized in $\Gamma$ and our claim is shown.
\end{claimlemmaproof}

\begin{claimlemma} \label{claim:taubipartite}
For $1 \leq k < \delta$, the triangle type $(1,k,k+1)^{\rho^{-1}}$ is in $\Gamma$.
\end{claimlemma}

\begin{claimlemmaproof}

By definition, $(1,k,k+1)^{\rho^{-1}}$ is either $(\delta,k/2,\delta-k/2)$ or $(\delta,\frac{k-1}{2},\delta-\frac{k-1}{2})$. In either case, $(1,k,k+1)^{\rho^{-1}}$ is of geodesic type and therefore will indeed be realized in $\Gamma$.
\end{claimlemmaproof}
The lemma follows.
\end{proof}

\begin{lemma}\label{lemma:sufficiency:rhoinv}
Let $\Gamma$ be a metrically homogeneous graph of generic type with finite diameter $\delta$ with associated numerical parameters
\begin{align*}
    K_1 = \delta && K_2 = \delta && C = 3\delta+1 && C' = C+1
\end{align*}
Then $\Gamma^{\rho^{-1}}$ is metrically homogeneous.
\end{lemma}

\begin{proof}
Our reasoning proceeds as in the proof of the previous lemma. That is, we show that $\Gamma^{\rho^{-1}}$ is metrically homogeneous by verifying that the triangle type $(i,j,k)$ for $i + j < k \leq \delta$ is not realized in $\Gamma^{\rho^{-1}}$ while the triangle types $(1,k,k+1)$ for $1 \leq k < \delta$ are. 
We work instead with the graph $\Gamma$ and the images of these triples under $\rho$.

\begin{claimlemma} 
For $i + j < k \leq \delta$, the triangle type $(i,j,k)^{\rho}$ is not in $\Gamma$.
\end{claimlemma}

\begin{claimlemmaproof}
If $k \leq \delta/2$, then $(i,j,k)^{\rho} = (2i,2j,2k)$. Since ${2i + 2j < 2k}$, this triple is not realized in $\Gamma$. 

If $i \leq \delta/2 < j$, then 
\begin{equation*}
    (i,j,k)^{\rho} = (2i,2(\delta-j)+1,2(\delta-k)+1). 
\end{equation*}
Since $2i + 2(\delta-k)+1 < 2(\delta-j)+1$, this triple cannot be realized in $\Gamma$.

The remaining case to consider is $i,j \leq \delta/2 < k$. In this case, 
\begin{equation*}
(i,j,k)^{\rho} = (2i,2j,2(\delta-k)+1). 
\end{equation*}
The perimeter here is $2(\delta+i+j-k)+1$ which is odd and less than $2\delta+1$, and thus by Corollary \ref{Fact:realizeoddp} is also excluded from being realized in $\Gamma$.

Thus no such triangle type $(i,j,k)$ will be realized in $\Gamma$.
\end{claimlemmaproof}

\begin{claimlemma} 
For $1 \leq k < \delta$, the triangle type $(1,k,k+1)^{\rho}$ is in $\Gamma$.
\end{claimlemma}

\begin{claimlemmaproof}
If $k \neq \lfloor \delta/2 \rfloor$, then $(1,k,k+1)^{\rho}$ is of geodesic type and is realized in $\Gamma$. If $k = \lfloor \delta/2 \rfloor$, then 
\begin{equation*}
(1,k,k+1)^{\rho} = (2,2k,2(\delta-(k+1))+1)    
\end{equation*}
which has perimeter $2\delta+1$. We apply Corollary \ref{Fact:realizeoddp} to see that this triangle type must be realized in $\Gamma$.
\end{claimlemmaproof}
\end{proof}

\begin{lemma}\label{lemma:sufficiency:smalldeltatau1}
Let $\Gamma$ be a metrically homogeneous graph of generic type. Suppose that $\delta =3$ or $4$ and that the parameters $K_1,K_2,C,C'$ for $\Gamma$ are among those in the table below. Then $\Gamma^{\tau_1}$ is also metrically homogeneous.

\begin{table}[ht]
\centering
\begin{tabular}{ccccc|ccccc}
$\delta$&$K_1$&$K_2$&$C$&$C'$
&$\delta$&$K_1$&$K_2$&$C$&$C'$
\\
\hline
$3$&$\infty$&$0$&$7$&$10$
&$4$&$1$&$3$&$11$&$14$ \\
$3$&$1$&$2$&$10$&$11$
&$4$&$1$&$3$&$11$&$12$ \\
$3$&$1$&$2$&$9$&$10$
&$4$&$2$&$3$&$11$&$12$\\
$3$&$2$&$2$&$9$&$10$
&$4$&$2$&$3$&$11$&$14$\\
$3$&$2$&$2$&$10$&$11$\\
\end{tabular}
\end{table}

\end{lemma}

\begin{proof}
By Lemma \ref{Fact:ProveMHG} the structure $\Gamma^{\tau_1}$ will be a metrically homogeneous graph if it does not realize any triple violating the triangle inequality, and does realize all geodesic types of the form $(1,k,k+1)$ with $1 \leq k < \delta$.

Thus for $\delta = 3$, the triple $(1,1,3)$ must be omitted by $\Gamma^{\tau_1}$ and the triangle types $(1,1,2)$ and $(1,2,3)$ must be realized by $\Gamma^{\tau_1}$.

For $\delta = 4$, the triples $(1,1,3), (1,1,4)$ and $(1,2,4)$ must be omitted by $\Gamma^{\tau_1}$ while the triangle types $(1,1,2), (1,2,3), (1,3,4)$ must be realized by $\Gamma^{\tau_1}$.

This translates into the following restrictions for $\Gamma$ when $\delta =3$:
\begin{center}
\begin{tabular}{ c| c }
  \text{Forbidden} & \text{Realized} \\
  \hline
  (1,3,3) & (1,2,3), (2,3,3)
\end{tabular}
\end{center}

\noindent and the following restrictions for $\Gamma$ when $\delta =4$:
\begin{center}
\begin{tabular}{ c | c }
  \text{Forbidden} & \text{Realized} \\
  \hline
  (3,4,4), (1,4,4), (1,2,4) & (2,4,4), (1,3,4), (2,3,4)
\end{tabular}
\end{center}

Equivalently, setting aside geodesic types and triples violating the triangle inequality, we must verify that $\Gamma$ satisfies the following four conditions.
\begin{itemize}
\item $K_2 < \delta$;
\item $\Gamma$ realizes the triangle type $(\delta,\delta,2)$;
\item for $\delta = 4$, $\Gamma$ realizes the triangle type $(2,3,4)$;
\item for $\delta = 4$, $\Gamma$ does not realize the triangle type $(3,4,4)$.
\end{itemize}

Indeed, in every line in our table above, $K_2 < \delta$. In addition, whenever $\delta = 4$ we have $C = 11$ and therefore the triangle type $(3,4,4)$ will be forbidden. This disposes of the first and last conditions.

In every line in our table, $C_0 \geq 2\delta +4$, and therefore $\Gamma$ contains a triangle of perimeter $2\delta+2$. By Fact \ref{Fact:Realize:delta,delta,d}, $\Gamma_{\delta}$ realizes some distance $d \geq 2$. If $K_1 =1$, then Facts \ref{Fact:LocalAnalysis} and \ref{Fact:LocalAnalysis:K1le2} tell us that $\Gamma_{\delta}$ is connected and therefore contains a pair of vertices at distance $2$. If $K_1 > 1$, then Fact \ref{Fact:LocalAnalysis:d=2Connected} tells us that $\Gamma_{\delta}$ is connected by the edge relation $d(x,y) =2$. 
Thus in either case the distance $2$ occurs in $\Gamma_{\delta}$ and our second condition is satisfied.

Suppose then that $\delta = 4$. To conclude our proof, we must show that the triangle type $(2,3,4)$ is realized in $\Gamma$. Thus we must find vertices $u \in \Gamma_3$ and $v \in \Gamma_4$ such that $d(u,v) =2$.

In every line in our table, $K_2 = 3$, and therefore $\Gamma_3$ contains an edge. Fact \ref{Fact:LocalAnalysis} tells us then that $\Gamma_3$ is connected. Moreover we also see that $\Gamma_3$ has diameter at least $2$, since $\Gamma_2$ has diameter $4$. Thus $\Gamma_3$ realizes the distance $2$. Take then $v \in \Gamma_4$ and define $I_v$ to be the set of neighbors of $v$ in $\Gamma_3$. If $I_v = \Gamma_3$ then by homogeneity every vertex of $\Gamma_4$ is adjacent to every vertex of $\Gamma_3$ and $\Gamma_4$ has diameter at most $2$, a contradiction. So $\Gamma_3 \neq I_v$. 

Since $\Gamma_3$ is connected, there is a vertex $u \in \Gamma_3 \setminus I_v$ adjacent to some vertex $v' \in I_v$. It then follows that $d(u,v) =2$ and $u,v$ and the basepoint form the desired triangle.
\end{proof}


\begin{lemma} \label{lemma:sufficiency:taumod2}
Let $\Gamma$ be a bipartite metrically homogeneous graph of generic type with diameter $\delta \equiv \epsilon \pmod 2$ where $\epsilon = 0$ or $1$ and $C_0 =2 (\delta+\epsilon) +2$. Then $\Gamma^{\tau_{\epsilon}}$ is metrically homogeneous.
\end{lemma}

\begin{proof}
By Fact \ref{Fact:ProveMHG}, it suffices to check that the triangle types $(i,j,k)$ for $$i+j<k \leq \delta$$ are not in $\Gamma^{\tau_{\epsilon}}$ and the triangle types $(1,k,k+1)$ for $1 \leq k < \delta$ are in $\Gamma^{\tau_{\epsilon}}$. We work in $\Gamma$ with their images under $\tau_{\epsilon}$.


\begin{claimlemma} 
For $i + j < k \leq \delta$, the triangle type $(i,j,k)^{\tau_{\epsilon}}$ is not in $\Gamma$.
\end{claimlemma}

\begin{claimlemmaproof}
We observe that since $\delta \equiv \epsilon \pmod 2$, the permutation $\tau_{\epsilon}$ preserves parity. Therefore, if $i+j+k$ is odd then $(i,j,k)^{\tau_{\epsilon}}$ is forbidden, as there are no triangle types of odd perimeter in bipartite graphs.

If $i,j,$ and $k$ are all even, then they will all be fixed by $\tau_{\epsilon}$. So the triangle type $(i,j,k)$ is excluded from $\Gamma$.

Thus it remains to consider the case in which one of $i,j,k$ is even and the other two are odd.

If $i$ and $j$ are odd and $k$ is even, then 
\begin{equation*}
 j^{\tau_{\epsilon}} + k^{\tau_{\epsilon}} = \delta+\epsilon -j + k < \delta + \epsilon -i = i^{\tau_{\epsilon}}   
\end{equation*}
and once again by the triangle inequality the triple $(i,j,k)^{\tau_{\epsilon}}$ will not be realized in $\Gamma$.

Finally, suppose that $i$ is even and $j$ and $k$ are odd. In this case, 
\begin{equation*}
 i^{\tau_{\epsilon}} + k^{\tau_{\epsilon}} = \delta+\epsilon + i - k < \delta + \epsilon - j = j^{\tau_{\epsilon}}   
\end{equation*}
so again the triple $(i,j,k)^{\tau_{\epsilon}}$ violates the triangle inequality.
\end{claimlemmaproof}

\begin{claimlemma}
For $1 \leq k < \delta$, the triangle type $(1,k,k+1)^{\tau_{\epsilon}}$ is in $\Gamma$.
\end{claimlemma}

\begin{claimlemmaproof}

If $k$ is even, then 
\begin{equation*}
 (1,k,k+1)^{\tau_{\epsilon}} = (\delta+\epsilon-1,k,\delta+\epsilon-(k+1))   
\end{equation*}
which is of geodesic type and therefore is realized in $\Gamma$.

If $k$ is odd, then 
\begin{equation*}
 (1,k,k+1)^{\tau_{\epsilon}} = (\delta+\epsilon-1,\delta+\epsilon -k,k+1)   
\end{equation*}
We consider the two values of $\epsilon$ separately.

Suppose first that 
\begin{align*}
    \epsilon = 0.
\end{align*}

Then by hypothesis the graph $\Gamma$ is antipodal and the triangle type $(1,k,{k+1})^{\tau_{\epsilon}}$ is $(\delta-1,\delta-k,{k+1})$. Replacing one of the vertices $v$ of this triangle type by its antipodal vertex $v'$ yields the triangle type $(1,k,k+1)$, by the antipodal law (Fact \ref{fact:antipodallawantipodalgraphs}). Therefore the original triangle type must in be $\Gamma$ since this triangle is of geodesic type.


Now suppose
\begin{align*}
\epsilon =1. 
\end{align*}
In this case, $C_0 = 2\delta + 4$ and $(i,j,k)^{\tau_{\epsilon}} = (\delta,\delta-k+1,k+1)$. Moreover, we may apply Fact \ref{Fact:Realize:delta,delta,d}
to get that $\diam(\Gamma_{\delta}) =2$. Since $(\delta,\delta-k+1,k+1)$ is invariant under the substitution of $\delta - k$ for $k$, we may assume that $k \leq \delta/2$. Applying Lemma \ref{lemma:diameterofsubgraph} then, we get that $\diam(\Gamma_{\delta-k+1}) = 2k$. We may take $u,v \in \Gamma_{\delta-k+1}$ at distance $2k$ and $u',v' \in \Gamma_{\delta}$ at distance $k-1$ from $u$ and $v$ respectively. It follows that $d(u',v') \geq 2$ and hence $d(u',v') = 2$. This implies that $d(u,v') \leq k+1$. By the triangle inequality, $d(u,v') \geq k+1$, and therefore $d(u,v') = k+1$. The triangle formed by $u,v'$ and the basepoint has type $(\delta-k+1,\delta,k+1)$, and hence this triangle type is indeed in $\Gamma$.
\end{claimlemmaproof}
\end{proof}

\begin{lemma}\label{lemma:sufficiency:tau}
Let $\Gamma$ be a metrically homogeneous graph of generic type with diameter $\delta$, with the numerical parameters
\begin{align*}
K_1 = \left\lfloor \frac{\delta+\epsilon}{2} \right\rfloor && K_2 = \left\lceil \frac{\delta+\epsilon}{2} \right\rceil && C = 2(\delta+\epsilon) + 1 && C' = C+1.
\end{align*}
Then $\Gamma^{\tau_{\epsilon}}$ is metrically homogeneous.
\end{lemma}

\begin{proof}
By Fact \ref{Fact:ProveMHG}, it suffices to show that violations of the triangle equality are not in $\Gamma^{\tau_{\epsilon}}$ and geodesics of the form $(1,k,k+1)$ for $1 \leq k < \delta$ are in $\Gamma^{\tau_{\epsilon}}$. We work in $\Gamma$ with the images of these triples under $\tau_{\epsilon}$. We begin with the first point.

\begin{claimlemma} 
For $i + j < k \leq \delta$, the triple $(i,j,k)^{\tau_{\epsilon}}$ is not in $\Gamma$;
\end{claimlemma}

\begin{claimlemmaproof} 
For brevity, we refer to the parity of $\min(h,\delta+\epsilon-h)$ as the $(\delta+\epsilon)$-parity of $h$. Note that by definition $h$ and $h^{\tau_{\epsilon}}$ have the same $(\delta+\epsilon)$-parity. We break our argument into cases, based on the relative $(\delta+\epsilon)$-parities of $i$, $j$, and $k$. 

\noindent\textit{Case 1.} The $(\delta+\epsilon)$-parity of $i$ and $j$ are both even.

If the $(\delta+\epsilon)$-parity of $k$ is also even, then $(i,j,k)^{\tau_{\epsilon}} = (i,j,k)$, which violates the triangle inequality.

If the $(\delta+\epsilon)$-parity of $k$ is odd and $\delta+\epsilon -k \geq k$, then $(i,j,k)^{\tau_{\epsilon}}$ still violates the triangle inequality. 

If the $(\delta+\epsilon)$-parity of $k$ is odd and $\delta + \epsilon - k < k$, then $k > (\delta+\epsilon)/2$.

In that case, $\min(k,\delta+\epsilon-k) = \delta + \epsilon -k$ is odd. Then the perimeter of $(i,j,k)^{\tau_{\epsilon}}$ is $i+j+\delta+\epsilon-k < \delta + \epsilon$. Since $K_1 = \lfloor \frac{\delta+\epsilon}{2} \rfloor$, the perimeter of $(i,j,k)^{\tau_{\epsilon}}$ must be even. Thus $i^{\tau_{\epsilon}}$ and $j^{\tau_{\epsilon}}$ have opposite parity, and we assume without loss of generality that $i^{\tau_{\epsilon}}$ is even and $j^{\tau_{\epsilon}}$ is odd.

From this we infer that $(\delta+\epsilon)$ is odd, that $i$ is even and $i^{\tau_{\epsilon}} = i < \frac{\delta+\epsilon}{2}$, and that $j^{\tau_{\epsilon}} = j > \frac{\delta+\epsilon}{2}$. We therefore conclude that
\begin{align*}
i^{\tau_{\epsilon}} + k^{\tau_{\epsilon}} = i + \delta+\epsilon -k < i + \delta + \epsilon - (i+j) = \delta + \epsilon -j < j = j^{\tau_{\epsilon}}
\end{align*}
and hence the triple $(i,j,k)^{\tau_{\epsilon}}$ violates the triangle inequality.

\noindent\textit{Case 2.} The $(\delta+\epsilon)$-parity of $i$ and $j$ are both odd.

\begin{sloppypar}
If the $(\delta+\epsilon)$-parity of $k$ is even, then the perimeter of $(i,j,k)^{\tau_{\epsilon}}$ is ${2(\delta+\epsilon) + k -i -j> 2}$ which is forbidden by the $C, C'$ bounds.
\end{sloppypar}

Suppose the $(\delta+\epsilon)$-parity of $k$ is odd. In this case, the perimeter of $(i,j,k)^{\tau_{\epsilon}}$ is 
\begin{equation*}
    3(\delta+\epsilon)-(i+j+k) > 3(\delta+\epsilon)-2k \geq 2 K_2 + 2(\delta_{\epsilon}-k) = 2K_2 + 2k^{\tau_{\epsilon}}
\end{equation*}
and thus violates the condition associated to $K_2$.

Without loss of generality, the remaining case is the following.

\noindent\textit{Case 3.} The $(\delta+\epsilon)$-parities of $i$ and $j$ are even and odd respectively.

If the $(\delta+\epsilon)$-parity of $k$ is odd, then we have that
\begin{equation*}
i^{\tau_{\epsilon}} + k^{\tau_{\epsilon}} = i + (\delta+\epsilon) -k < (\delta+\epsilon)-j = j^{\tau_{\epsilon}}
\end{equation*}
and therefore this triple violates the triangle inequality.

If the $(\delta+\epsilon)$-parity of $k$ is even then the perimeter of $(i,j,k)^{\tau_{\epsilon}}$ is 
\begin{align*}
    i + (\delta+\epsilon -j) +k \geq (\delta + \epsilon) + 2i \geq 2K_2 + 2i
\end{align*}
and the bound associated with $K_2$ is violated.

We may finally assume that 
\begin{align*}
i > \frac{\delta+\epsilon}{2}   &&  j \leq \frac{\delta+\epsilon}{2}  && k > \frac{\delta+\epsilon}{2}.
\end{align*}

Then the perimeter of $(i,j,k)^{\tau_{\epsilon}}$ is
\begin{equation*}
i + (\delta+\epsilon)-j+k \geq (\delta+\epsilon) + 2i > 2(\delta+\epsilon)
\end{equation*}
which violates the $C$ bounds.\\

We conclude then that every such triple $(i,j,k)^{\tau_{\epsilon}}$ is forbidden is $\Gamma$. 
\end{claimlemmaproof}

\begin{claimlemma} 
For $1 \leq k < \delta$ the triangle type $(1,k,k+1)^{\tau_{\epsilon}}$ is in $\Gamma$.
\end{claimlemma}

\begin{claimlemmaproof} Note that $\min(k,\delta+\epsilon-k) \leq \lfloor \frac{\delta+\epsilon}{2} \rfloor$.

If $\min(k,\delta+\epsilon-k) = \lfloor \frac{\delta+\epsilon}{2} \rfloor$ then $(1,k,k+1)^{\tau_{\epsilon}} = (1,k,k+1)$, with $k$ and $k+1$ either fixed for swapped. So we may assume 
\begin{align*}
    \min(k,\delta+\epsilon-k) < (\delta+\epsilon)/2.
\end{align*}

If $\min(k,\delta+\epsilon-k)$ is even, then $(1,k,k+1)^{\tau_{\epsilon}} = (\delta+\epsilon-1,k,\delta+\epsilon-k-1)$. This is itself of geodesic type, and therefore must be realized in $\Gamma.$

Now suppose that $\min(k,\delta+\epsilon-k)$ is odd. If $\epsilon = 0$, then $\Gamma$ is antipodal and the triangle type under consideration is $(k+1,\delta-k,\delta-1)$.



We may therefore replace one of the vertices of this triangle types with its opposite to obtain the triangle type $(1,k,k+1)$. Since the latter triangle type is realized in $\Gamma$, the former one is as well.

If $\epsilon =1$, then we require the triangle type $(\delta,\delta-k+1,k+1)$ to be realized in $\Gamma$. We follow the proof of Claim \ref{claim:taubipartite} in the proof of Lemma \ref{lemma:sufficiency:taumod2}. We argue that $\Gamma_{\delta}$ has diameter $2$, that we may suppose without loss of generality that $k \leq \delta/2$, and that there are vertices $u,v \in \Gamma_{\delta-k+1}$ at distance $2k$. Then taking $u',v' \in \Gamma_{\delta}$ at distance $k-1$ from $u,v$ respectively, we find $d(u',v') = 2$ and argue as previously that the triangle type formed by $u,v'$ and the basepoint has the desired type. This proves the claim.
\end{claimlemmaproof}
These two claims show that $\Gamma^{\tau_{\epsilon}}$ is indeed a metrically homogeneous graph.
\end{proof}

\begin{proof}[Proof of Proposition \ref{prop:sufficiency}]

This follows immediately from Lemmas \ref{lemma:sufficiency:rho}, \ref{lemma:sufficiency:rhoinv}, \ref{lemma:sufficiency:smalldeltatau1} \ref{lemma:sufficiency:taumod2}, and \ref{lemma:sufficiency:tau}.
\end{proof}

\section{Appendix: Structure of metrically homogeneous graphs}\label{sec:app}


Below is a compilation of relevant facts about the structure of metrically homogeneous graphs. Previously known results are referenced and labeled as ``facts," while the new results are labeled as a ``proposition" and a ``lemma."

We assume familiarity with the terminology and notation of Section \ref{sec:basics}.


\subsection{Known metrically homogeneous graphs}
\begin{fact}[{\cite[Theorem 10, Lemma 8.6]{Che-2P}}]\label{Fact:NongenericType}
Let $\Gamma$ be a metrically homogeneous graph of diameter $\delta\geq 2$
(in particular, $\Gamma$ is connected).
Then one
of the following applies.
\goodbreak
\begin{itemize}
\item $\Gamma$ is finite.
\item $\Gamma$ is complete multipartite with at least two classes (this includes
the case in which $\Gamma$ is complete).
\item $\Gamma$ is the complement $H_n^c$ 
of a Henson graph with $3\leq n<\infty$.
\item $\Gamma$ is one of the tree-like graphs $T_{m,n}$, of infinite diameter.
\item $\Gamma$ is of generic type with diameter at least 2.
\end{itemize}

In particular, if $\Gamma$ is infinite and of finite diameter $\delta\geq 3$, 
then $\Gamma$ is of generic type.
\end{fact}

\begin{fact}\label{Fact:Finite:delta>=3}
The finite metrically homogeneous graphs of diameter at least $3$
are of the following two forms.
\begin{enumerate}
\item An $n$-cycle with $n\geq 6$ $(\delta=\lfloor n/2 \rfloor)$.
\item Diameter 3, antipodal double cover of one of the following.
\begin{enumerate}
\item  $C_5$, 
\index{graph!pentagon}
\item $K_3 \otimes K_3$, or
\item $I_n$ (an independent set of order $n\geq 2$).
\end{enumerate}
\end{enumerate}
\end{fact}

The following result is given in \cite{ACM-MH3} as Theorem 1, in the form of a completely explicit catalog.
\begin{fact}[{\cite[Theorem 1]{ACM-MH3}}]\label{Fact:MH3}
The metrically homogeneous graphs of diameter 3 are all of known type, that is, either finite or of the form $\Gamma^{\delta}_{K_1,K_2,C,C'}$ with Henson constraints $\mathcal{S}$.
\end{fact}

\begin{fact}[{\cite[Lemma 2.7]{Che-MH4}}]\label{Fact:evenPdiam4}
Let $\Gamma$ be a metrically homogeneous graph of diameter $4$, with associated numerical parameters $K_1,K_2,C_0,C_1$. Then any triangle of even perimeter $p < C_0$ embeds isometrically into $\Gamma$.
\end{fact}

\subsection{Local analysis}

We have the following for antipodal graphs.

\begin{fact}[{\cite[Lemma 6.1]{Che-2P}}]\label{fact:antipodalSmallC}
Let $\Gamma$ be a metrically homogeneous graph of diameter $\delta$. Then $\Gamma$ is antipodal if and only if no triangle has perimeter greater than $2\delta$.
\end{fact}

\begin{fact}[{\cite[Theorem 11]{Che-2P}}]\label{fact:antipodallawantipodalgraphs}
Let $\Gamma$ be a metrically homogeneous graph and antipodal of diameter $\delta \geq 3$. Then for each $u \in \Gamma$ there exists a $u' \in \Gamma$ at distance $\delta$ from $u$ such that the antipodal law
\begin{equation*}
    d(u,v) = \delta - d(u',v)
\end{equation*}
holds for every $v \in \Gamma$.

In particular, the map $u \mapsto u'$ is a central involution of $\Aut(\Gamma)$.
\end{fact}

The following rephrases
Proposition 5.1 of \cite{Cam-Cen}, where the statement is given in terms of the first two of our three conditions, and in greater generality (for the distance transitive case).

\begin{fact}
\label{Fact:CharacterizeMHG}
Let $(\Gamma,d)$ be a homogeneous integer-valued metric space
and let $(\Gamma,E)$ be $\Gamma$ viewed as a graph with edge relation
$d(x,y)=1$.

Then the following are equivalent.
\begin{enumerate}
\item $(\Gamma,E)$ is a metrically homogeneous graph, and $d$ is the
graph metric.
\item $(\Gamma,E)$ is connected.
\item $(\Gamma,d)$ contains all triangles of type $(1,k,k+1)$ with $k$ less than 
the diameter of $(\Gamma,d)$.
\end{enumerate}
\end{fact}

This may be phrased more usefully for our purposes as follows.
\begin{fact}
\label{Fact:ProveMHG}
Let $\Gamma$ be a metrically homogeneous graph and $\sigma$
a permutation of the language. Then the following are equivalent.
\begin{itemize}
\item $\Gamma^\sigma$ is a metrically homogeneous graph
\item $\Gamma^\sigma$ is a metric space, and contains all triangles of
type $(1,k,k+1)$ for $k$ less than the diameter of $\Gamma$.
\end{itemize}
\end{fact}

\begin{fact}[{\cite[Lemma 15.6]{Che-HOGMH}}]
\label{Fact:Realize:delta,delta,d}
Let $\Gamma$ be a metrically homogeneous graph of diameter $\delta$
which contains a triangle of perimeter $2\delta+d$. Then $\Gamma_{\delta}$
has diameter at least $d$.
\end{fact}

\begin{fact}[{\cite[Lemma 15.5]{Che-HOGMH}}]
\label{Fact:Neighbors:i+-1}
Let $\Gamma$ be a metrically homogeneous graph of generic type.
Suppose $1\leq i\leq \delta$.
Then for $u\in \Gamma_{i\pm1}$, the set $\Gamma_1(u) \cap \Gamma_i$ is
infinite. 
\end{fact}

\begin{fact}[{\cite[Lemma 15.4]{Che-HOGMH}}]
\label{Fact:LocalAnalysis:d=2Connected}

Let $\Gamma$ be a metrically homogeneous graph of generic type 
and diameter $\delta$.
Suppose $i\leq \delta$, and suppose also that if $i=\delta$ then $K_1>1$.
Then the metric space $\Gamma_i$ is connected with respect to the edge relation
defined by 
$$d(x,y)=2$$
\end{fact}

\begin{fact}[{\cite[Thm.~1.29]{Che-HOGMH}}]\label{Fact:LocalAnalysis}

Let $\Gamma$ be a metrically homogeneous graph of generic type and of 
diameter $\delta$, and suppose $i\leq \delta$. 
Suppose that $\Gamma_i$ contains an edge.
Then $\Gamma_i$ is a metrically homogeneous graph (and, in particular, is connected).

\goodbreak
Furthermore, $\Gamma_i$ is primitive and of generic type, apart from the following two cases.

\goodbreak
\begin{enumerate}
\item $i=\delta$;  \\
$K_1=1$; $\{C_0,C_1\}=\{2\delta+2,2\delta+3\}$;\\ 
$\Gamma_\delta$ is an infinite complete graph 
(hence not of generic type). 
\item $\delta=2i$; \\ 
$\Gamma$ is antipodal (hence $\Gamma_i$ is imprimitive, namely antipodal).
\end{enumerate}

\end{fact}

The following fact is technical, though quite helpful. 
\begin{lemma}\label{Fact:C:delta'}
Let $\Gamma$ be a metrically homogeneous graph of diameter $\delta$
and let ${\delta'=\diam(\Gamma_{\delta})}$.

If $\Gamma$ is bipartite then $C=2\delta+1$ and $C'=2\delta+\delta'+2$.

If $\Gamma$ is not bipartite and one of the following conditions $(a,b)$ applies,
then $C'$ is $2\delta+\delta'+2$, and either $C$ is $2\delta+\delta'+1$ or $\delta'=2$ and $C= 2\delta+1$.
\begin{enumerate}[label=(\alph*)]
\item $K_2=\delta$; or
\item The set 
\begin{align*}
D_\delta&=\{d \hspace{1mm}|\hspace{1mm} \mbox{$d>0$ and there is a triangle of type $(\delta,\delta,d)$ in $\Gamma$}\}
\end{align*}
is an interval. 
\end{enumerate}
\end{lemma}

\begin{proof}
We consider first the case when $\Gamma$ is antipodal. Then $\Gamma_{\delta}$ consists of a single vertex, and in particular $\diam(\Gamma_{\delta}) = 0$. By Fact \ref{fact:antipodalSmallC}, $C = 2\delta+1$ and $C' = 2\delta+2$, and our assertion holds.

We turn then to the case when $\Gamma$ is not antipodal, and therefore there are at least two distinct vertices in $\Gamma_{\delta}$. From Fact \ref{Fact:Realize:delta,delta,d}, no triangle of perimeter $2\delta + d$ for $d > \delta'$ is realized in $\Gamma$. Thus 
\begin{align*}
    C \leq 2\delta + \delta' + 1 && C' \leq 2\delta + \delta' + 2
\end{align*}

Let $d_0 > 0$ denote the minimum distance between two distinct vertices in $\Gamma_{\delta}$. We observe that $d_0 \leq 2$. Indeed, given a vertex $w \in \Gamma_{\delta-1}$, Fact \ref{Fact:Neighbors:i+-1} tells us that there are two distinct vertices $u,v$ in $\Gamma_{\delta}$ and adjacent to $w$. Thus $d(u,v) \leq 2$; homogeneity yields that $d_0 \leq 2$.

If $\Gamma$ is bipartite, then $C = 2\delta + 1$. Moreover, the set $D_{\delta}$ consists solely of even numbers and therefore $\delta' = \max D_{\delta}$ is even. Thus $C' = C_0 \geq 2\delta+\delta'+2$. As we have already shown that $C' \leq 2\delta+\delta'+2$, we have that $C' = 2\delta+\delta'+2$.

Thus we turn our attention to the case
\begin{align*}
    \text{$\Gamma$ neither antipodal nor bipartite}
\end{align*}

We first suppose $$K_2 = \delta.$$

Then fact \ref{Fact:LocalAnalysis} tells us that $D_{\delta}$ is the interval from $1$ to $\delta'$. Hence $C \geq 2\delta+\delta'+1$. Our above inequalities therefore yield that $C = 2\delta+\delta'+1$ and $C' = 2\delta+\delta'+2$. 

We finally consider the case
\begin{align*}
    \text{$\Gamma$ antipodal and not bipartite.}
\end{align*}

Since $K_2 < \delta$, the set $D_{\delta}$ is the interval from $d_0$ to $\delta'$. We argued above that $d_0 \leq 2$. The assumption $K_2 < \delta$ yields that $d_0 \neq 1$, and therefore $D_{\delta}$ is the interval from $2$ to $\delta'$.

Let $v_*$ be a basepoint of $\Gamma$ and $a \in \Gamma_1$, $b \in \Gamma_{\delta}$. Then $d(a,b) \geq \delta-1$ and therefore the triangle $(v_*,a,b)$ has type either $(1,\delta-1,\delta)$ or $(1,\delta,\delta)$. As $K_2 < \delta$, this type is $(1,\delta-1,\delta)$ and therefore $d(a,b) = \delta-1$.

There exists a pair of vertices in $\Gamma_{\delta}$ at distance $\delta'$ from each other. Thus homogeneity ensures that we can find $b' \in \Gamma_{\delta}$ where $d(b,b') = \delta'$. The triangle $(a,b,b')$ has type $(\delta-1,\delta-1,\delta')$. This triangle has perimeter $2\delta+\delta'-2$, which is larger than $2\delta$ when $\delta' \geq 3$. Thus, using our above inequalities, we deduce that when $\delta' \geq 3$, we have  $C = 2\delta + \delta' + 1$ and $C' = 2\delta + \delta' + 2$.

To conclude, we treat the case
\begin{equation*}
    d_0 = \delta' = 2.
\end{equation*}
Then our above inequalities become
\begin{align*}
    C \leq 2\delta + 3 && C' \leq 2\delta + 4.
\end{align*}
Since $d_0 = 2$ we have that $C_0 \geq 2\delta + 4$. Thus $C' = C_0 = 2\delta + 4$ and ${C \in \{2\delta+1, 2\delta+3\}}$, as desired.
\end{proof}

\begin{fact}[{\cite[Prop.~1.30]{Che-HOGMH}}]\label{Fact:LocalAnalysis:K1le2}
Let $\Gamma$ be a metrically homogeneous graph of diameter $\delta$.
Suppose 
$$K_1\leq 2$$ 
Then for $2\leq i\leq \delta-1$, $\Gamma_i$ contains an edge,
unless $i=\delta-1$, $K_1=2$,  and $\Gamma$ is antipodal.
\end{fact}

\begin{fact}[{\cite[Lemma 13.15]{Che-HOGMH}}]\label{Fact:K1}
Let $\Gamma$ be a metrically homogeneous graph containing some 
triangle of odd perimeter, and let $p$ be the least (odd) number which is
the perimeter of such a triangle. Then the following hold.
\begin{enumerate}
\item A $p$-cycle embeds isometrically in $\Gamma$.
\item $p\leq 2\delta+1$.
\item $p=2K_1+1$.
\end{enumerate}
\end{fact}

This fact implies the following.
\begin{corfact}\label{Fact:realizeoddp}
Let $\Gamma$ be a metrically homogeneous graph with $K_1 < \infty$. Then the following hold.
\begin{enumerate}
    \item Any triangle type of odd perimeter $p < 2 K_1 + 1$ is not realized in $\Gamma$.
    \item Any triangle type of perimeter $2 K_1 + 1$ must be realized in $\Gamma$.
\end{enumerate}
\end{corfact}

\subsection{Extension of the theory of metrically homogeneous graphs}
The following is new. 
\begin{lemma} \label{lemma:diameterofsubgraph}
Let $\Gamma$ be a metrically homogeneous graph of diameter $\delta$ and let $\delta'$ be the diameter of $\Gamma_{\delta}$. Then for $i \leq (\delta-\delta')/2$, the diameter of $\Gamma_{\delta-i}$ is $\delta' + 2i$.
\end{lemma}

\begin{proof}
This is proved by induction on $i$, with the base case $i=0$ holding by definition. 

Suppose $i > 0, \delta'+2i \leq \delta$, and $\Gamma_{\delta-j}$ has diameter $\delta'+2j-2$ for $j < i$. We show first that the diameter of $\Gamma_{\delta-i}$ is at least $\delta'+2i$.

Take $u,v \in \Gamma{\delta-(i-1)}$ at distance $\delta'+2i-2$
and $u',v'$ adjacent to $u,v$ respectively at distance $d(u,v)+2$. 
It suffices to show that $u',v'$ are both in $\Gamma_{\delta-i}$ to conclude. Note that $u',v' \in \{\Gamma_{\delta-(i-2)},\Gamma_{\delta-(i-1)},\Gamma_{\delta-i}\}$.
Assume towards a contradiction that we do not have $u',v' \in \Gamma_{i-1}$. Our inductive diameter bounds imply that we do not have $u',v' \in \Gamma_{\delta-(i-1)}$, nor do we have $u',v' \in \Gamma_{\delta-(i-2)}$. 

Suppose then that $u' \in \Gamma_{\delta-(i-1)}, v' \in \Gamma_{\delta-(i-2)}$. In that case, we take $v'' \in \Gamma_{\delta-(i-1)}$ adjacent to $v'$ and then $d(u',v'')$ violates the diameter bound of $\Gamma_{\delta-(i-2)}$. Thus we may assume that $u' \in \Gamma_{\delta-i}$. 

This leaves us with $v' \in \Gamma_{\delta - j}$ with $j = i-1$ or $i -2$. We therefore take $u'' \in \Gamma_{\delta-j}$ with $d(u',u'') = i-j$, and find
\begin{equation*}
    d(u'',v') \geq d(u',v') - d(u',u'') = d(u,v)+2-(i-j) = \delta' + i + j > \delta + 2j,
\end{equation*}
which is a violation of the diameter bound on $\Gamma_{\delta-j}$. Thus we have a contradiction. Hence, $u',v' \in \Gamma_{\delta-(i-j)}$.

This shows that the diameter of $\Gamma_{\delta-i}$ is at least $\delta'+2i$, and the reverse inequality follows similarly: if $u,v \in \Gamma_{\delta-i}$ and $u',v'$ are adjacent to $u,v$ respectively and lie in $\Gamma_{\delta-i+1}$, induction and the triangle inequality give ${d(u,v) \leq \delta'+2i}$.
\end{proof}

\begin{prop}[Distances realized in $\Gamma_i$] \label{Prop:Realize:i,i,2k}
Let $\Gamma$ be a metrically homogeneous graph of generic type.
Suppose $k\leq i$ and $i+k\leq \delta$.
Then $\Gamma$ contains a triangle of type $(i,i,2k)$.
\end{prop}

We begin the proof of Proposition \ref{Prop:Realize:i,i,2k} with a lemma.

\begin{lemma}\label{Lemma:GeodesicPath}
Let $\Gamma$ be a metrically homogeneous graph of generic type
and diameter $\delta$. 
Suppose $k\leq \delta/2$.

Then there is a geodesic path $(v_0,\dots,v_k)$ in $\Gamma_k$ with
\begin{align*}
d(v_\ell,v_{\ell+1})&=2 \mbox{ for $0\leq \ell<k$}
\end{align*}
and there is a vertex $u\in \Gamma_{2k}$ such that $d(u,v_\ell)=k$ for all
$\ell\leq k$.
\end{lemma}

\begin{proof}
We now show by induction on $k\geq 1$ that the desired geodesic path exists in $\Gamma_k$. We will subsequently find a vertex $u$ with the desired property. 

When $k =1$, the lemma follows from the definition of generic type.
Now suppose $k\geq 2$ and that such a path of length $j$ exists in $\Gamma_j$ for all $j < k$.

As $2k \leq \delta$ and $\Gamma$ is assumed to be connected, there are points $v_0,v_k$ at distance $2k$ in $\Gamma$. There is therefore also a point at distance $k$ from both $v_0$ and $v_k$,
which we may take as the basepoint for $\Gamma$. Then $v_0,v_k\in \Gamma_k$.

Take $a,b$ in $\Gamma_{k-1}$ adjacent to $v_0,v_k$ respectively.
Observe that ${d(a,b) = 2(k-1)}$.

By our induction hypothesis, we have that there is a geodesic path
$(w_0,\dots,w_{k-1})$ in $\Gamma_{k-1}$ with $d(w_\ell,w_{\ell-1})=2$ for $\ell<k-1$. By homogeneity we may suppose that $w_0=a$ and $w_{k-1}=b$.

We claim that for $1\leq \ell\leq k-1$ there is $v_\ell\in \Gamma_k$ adjacent to $w_{\ell-1}$ and to $w_\ell$.

This holds because $\Gamma$ is of generic type. If we take $v\in \Gamma_k$ and $v'\in \Gamma_{k-2}$ at distance $2$ from $v$, then 
as $\Gamma$ is of generic type, there is a pair of points $w,w'$ in $\Gamma_{k-1}$ where both are adjacent to $v,v'$  and $d(w,w')=2$. Since $w$ and $w'$ have the common neighbor $v$ in $\Gamma_k$, the claim follows by homogeneity.

Now we consider the metric path $(v_0,v_1,\dots,v_k)$. Then for $\ell<k$, 
the vertices $v_\ell$, $v_{\ell+1}$ have the common neighbor $w_\ell$, and thus $d(v_\ell, v_{\ell+1})\leq 2$. On the other hand $d(v_0,v_k)=2k$,
and it follows that $d(v_\ell,v_{\ell+1})=2$, 
and that this is a geodesic path in 
$\Gamma_k$.

Thus the desired geodesic path exists.
To complete the proof of the lemma we need to find $u\in \Gamma_{2k}$
at distance $k$ from the vertices $v_\ell$ ($\ell\leq k$).

In the first part of the proof we constructed successively longer paths in $\Gamma_i$ for $i\leq k$. Now we construct successively shorter paths in 
$\Gamma_{k+i}$ for $0 \leq i\leq k$.

Set $v_\ell^0=v_\ell$ for $\ell \leq k$. For $i\leq k$ we construct geodesic paths
$(v_0^i,\dots,v_{k-i}^i)$ in $\Gamma_{k+i}$ inductively,
taking $v_\ell^i$ to be a common neighbor in $\Gamma_{k+i}$ of the 
vertices $v_\ell^i, v_{\ell+1}^i$, as in the claim above.

We then have $d(v_\ell^i,v_{\ell+1}^i)\leq 2$ for $\ell<k-i$, since each such pair has a common neighbor in the previous geodesic path. In particular
$d(v_0^i,v_{k-i}^i)\leq 2(k-i)$. But $d(v_0^{i-1},v_{k-(i-1)}^{i-1})=2(k-(i-1))$ and 
$v_0^{i-1}, v_{k-(i-1)}^{i-1}$ are adjacent to $v_0^i$, $v_{k-i}^i$, respectively, so $d(v_0^i,v_{k-i}^i)\geq 2(k-i)$. It follows that $(v_0^i,\dots,v_{k-i}^i)$ is again a geodesic path.

When $i=k$ our geodesic path degenerates to a point $u=v_0^i$ in $\Gamma_{2k}$.  By the construction, $u$ is connected by a path of length $k$ to each point $v_\ell^0$, and thus $d(u,v_\ell)\leq k$
for all $\ell \leq k$. But as $v_\ell\in \Gamma_k$ and $u\in \Gamma_{2k}$, we find $d(u,v_\ell)=k$, as required.
\end{proof}

\begin{proof}[Proof of Proposition \ref{Prop:Realize:i,i,2k}]
We have $k\leq i$, $i+k\leq \delta$. In particular $\Gamma_{i-k}$
and $\Gamma_{i+k}$ are both well-defined.

As $2k\leq i+k\leq \delta$, there is a geodesic triangle of type $(i-k,2k,i+k)$
in $\Gamma$. 
Therefore we may fix
$a\in \Gamma_{i-k}$ and $b\in \Gamma_{i+k}$ with
$d(a,b)=2k$.

Applying the previous lemma with basepoint $a$ and with $b$ playing the role of $u$ there,
by homogeneity 
we may find a geodesic path $(v_0,\dots,v_k)$ in $\Gamma_k(a)$ with
\begin{align*}
d(v_\ell,v_{\ell+1})&=2 \mbox{ $(\ell<k)$} \mbox{ and }  d(b,v_\ell)=k \mbox{ $(\ell \leq k)$}
\end{align*}

As $d(a,v_\ell)=d(b,v_\ell)=k$ for $\ell\leq k$, and $a \in \Gamma_{i-k}, b \in \Gamma_{i+k}$,
it follows that $v_\ell\in \Gamma_i$. In particular the vertices $v_0,v_k$
and the basepoint of $\Gamma$ form a triangle of type $(i,i,2k)$.
\end{proof}

\newpage

\section*{References}

\bibliography{Rebecca_Coulson_TwistsandTwistability_Feb012018.bbl} 

\end{document}